\documentclass[11pt]{amsart}
\usepackage[centertags]{amsmath}
\usepackage{amsfonts}
\usepackage{amssymb}
\usepackage{amsthm}
\usepackage{newlfont}
\usepackage{dsfont}
\usepackage{amscd}
\usepackage{mathrsfs}
\usepackage[all,cmtip]{xy}
\usepackage{tikz}
\usetikzlibrary{trees}
\usepackage[pagebackref=false,colorlinks]{hyperref}
\hypersetup{pdffitwindow=true,linkcolor=blue,citecolor=blue,urlcolor=cyan}
\usepackage{cite}
\usepackage{fancyhdr}
\usepackage{stmaryrd}
\usepackage{xcolor}

\usepackage[OT2,OT1]{fontenc}
\newcommand\cyr{%
\renewcommand\rmdefault{wncyr}%
\renewcommand\sfdefault{wncyss}%
\renewcommand\encodingdefault{OT2}%
\normalfont
\selectfont}
\DeclareTextFontCommand{\textcyr}{\cyr}

\newcommand{\mint}{{\times}\kern-0.89em{\int}} 

\usepackage[toc, page] {appendix}

\newcommand{\bdp}{\mathrm{BDP}}

\newcommand{\sha}{\textrm{{\cyr SH}}}

\newcommand{\bQ}{\mathbf{Q}}
\newcommand{\bZ}{\mathbf{Z}}

\newcommand{\pp}{{\wp}}

\usepackage{fullpage}

\newtheorem{thm}{Theorem}[section]
\newtheorem{cor}[thm]{Corollary}
\newtheorem{lem}[thm]{Lemma}
\newtheorem{prop}[thm]{Proposition}

\newtheorem{conj}[thm]{Conjecture}
\newtheorem{intro-conj1}[thm]{Conjecture}
\newtheorem*{assumption-spade}{Hypothesis~{$\spadesuit$}}
\newtheorem*{assumption-heart}{Hypothesis~{$\heartsuit$}}
\newtheorem*{assumption-CR}{Condition~CR}
\newtheorem*{ThmA}{Theorem A}
\newtheorem*{ThmB}{Theorem B}

\theoremstyle{definition}
\newtheorem{defn}[thm]{Definition}
\newtheorem{rem}[thm]{Remark}

\begin{document}
\title{A proof of Perrin-Riou's Heegner point main conjecture}
\author{Ashay Burungale}
\address[Ashay A. Burungale]{Department of Mathematics, California Institute of Technology, 1200 E California Blvd, 
CA 91125, USA}
\email{ashay@caltech.edu}
\author{Francesc Castella}
\address[Francesc Castella]{Department of Mathematics, University of California, Santa Barbara, CA 93106, USA}
\email{castella@ucsb.edu}
\author{Chan-Ho Kim}
\address[Chan-Ho Kim]{Center for Mathematical Challenges,  Korea Institute for Advanced Study, 85 Hoegiro, Dongdaemun-gu, 
Seoul 02455, Republic of Korea} 
\email{chanho.math@gmail.com}
\subjclass[2010]{11R23 (Primary); 11F33 (Secondary)}
\keywords{Iwasawa theory, Heegner points, Euler systems, $p$-adic $L$-functions}
\maketitle

\begin{abstract}
Let $E/\bQ$ be an elliptic curve of conductor $N$, let $p>3$ be a prime where $E$ has good ordinary reduction, and let $K$ be an imaginary quadratic field satisfying the Heegner hypothesis. 
In 1987, Perrin-Riou formulated an Iwasawa main conjecture for the Tate--Shafarevich group of $E$ over the anticyclotomic $\bZ_p$-extension of $K$ in terms of Heegner points. 

In this paper, we give a proof of Perrin-Riou's conjecture under mild hypotheses. Our proof builds on Howard's theory of bipartite Euler systems and Wei~Zhang's work on Kolyvagin's conjecture. In the case when $p$ splits in $K$, we also obtain a proof of the Iwasawa--Greenberg main conjecture for the $p$-adic $L$-functions of Bertolini--Darmon--Prasanna. 
\end{abstract}

\setcounter{tocdepth}{1}
\tableofcontents

\section{Introduction}

\subsection{The Heegner point main conjecture}

Let $E/\mathbf{Q}$ be an elliptic curve of conductor $N$, and let $p>3$ be a prime where $E$ has good ordinary reduction. Let $K$ be an imaginary quadratic field of discriminant $D_K<0$ prime to $Np$. Throughout the paper, we assume that 
\begin{equation}\label{assu:disc}
\textrm{$D_K$ is odd, and $D_K\neq -3$.}\tag{disc}
\end{equation}
Write $N$ as the product
\[
N=N^+N^-
\]
with $N^+$ (resp. $N^-$) divisible only by primes that split (resp. remain inert) in $K$, and assume  the following \emph{generalized Heegner hypothesis}:  
\begin{equation}\label{assu:gen_heeg}
\textrm{$N^-$ is the squarefree product of an even number of primes.}\tag{Heeg}
\end{equation}
 
Under this hypothesis, exploiting the modularity of $E$, \cite{bcdt}, for every positive integer $n$ prime to $N$ the theory of complex multiplication yields a construction of Heegner points $y_n\in E(H_n)$ defined over the ring class field $H_n$ of $K$ of conductor $n$. More precisely, letting $X_{N^+,N^-}$ be the Shimura curve attached to an indefinite quaternion algebra $B/\mathbf{Q}$ of discriminant $N^-$ together with a $\Gamma_0(N^+)$-level structure, the points $y_n$ are obtained as the image of special points on $X_{N^+,N^-}$ under a fixed  parametrization $\pi:X_{N^+,N^-}\rightarrow E$.

Let $K_\infty$ be the anticyclotomic $\mathbf{Z}_p$-extension of $K$, and for any number field $L$ let $\mathrm{Sel}_{p^\infty}(E/L)$ and $S_p(E/L)$ be the Selmer groups of $E/L$ fitting into the descent exact sequences
\[
0\rightarrow E(L)\otimes\mathbf{Q}_p/\mathbf{Z}\rightarrow\mathrm{Sel}_{p^\infty}(E/L)\rightarrow\sha(E/L)[p^\infty]\rightarrow 0,
\]
\[
0\rightarrow E(L)\otimes\mathbf{Z}_p\rightarrow S_p(E/L)\rightarrow\varprojlim_j\sha(E/L)[p^j]\rightarrow 0.
\] 
The study of anticyclotomic Iwasawa theory for elliptic curves was initiated by Mazur \cite{mazur-icm}, who conjectured in particular that the Pontryagin dual
\[
X_{\infty}:=\mathrm{Hom}_{\mathbf{Z}_p}(\varinjlim_n\mathrm{Sel}_{p^\infty}(E/K_n),\mathbf{Q}_p/\mathbf{Z}_p)
\]
has rank one over the Iwasawa algebra $\Lambda=\mathbf{Z}_p \llbracket \mathrm{Gal}(K_\infty/K )\rrbracket$. 
Under the $p$-ordinarity hypothesis, the Kummer images of Heegner points give rise to a compatible system of classes 
\[
\kappa_{\infty}\in S_\infty:=\varprojlim_n S_p(E/K_n),
\]
and it was also conjectured by Mazur that $S_\infty$ has $\Lambda$-rank one and the class $\kappa_\infty$ is not $\Lambda$-torsion. In this context, Perrin-Riou \cite{perrin-riou-heegner} (in the case $N^-=1$, and later extended by Howard \cite{howard-gl2-type} to allow $N^-\neq 1$) formulated the following variant of the Iwasawa main conjecture.

\begin{intro-conj1}[{\bf Heegner point main conjecture}]\label{conj:HPMC}
Suppose $K$ satisfies hypotheses $\mathrm{(\ref{assu:disc})}$ and $\mathrm{(\ref{assu:gen_heeg})}$.  
Then $S_\infty$ and $X_\infty$ have both $\Lambda$-rank one, and there is a finitely generated torsion $\Lambda$-module $M_\infty$ for which:
\begin{itemize}
\item[(i)] There is a $\Lambda$-module pseudo-isomorphism $X_\infty\sim\Lambda\oplus M_\infty\oplus M_\infty$. 
\item[(ii)] The characteristic ideal of $M_\infty$ satisfies $\mathrm{Char}_\Lambda(M_\infty)=\mathrm{Char}_\Lambda(M_\infty)^\iota$  and 
\[
\mathrm{Char}_\Lambda(M_{\infty})=
\mathrm{Char}_\Lambda\bigl(S_\infty/\Lambda\kappa_\infty\bigr),
\]
where $\iota:\Lambda\rightarrow\Lambda$ is the involution given by $\gamma\mapsto\gamma^{-1}$ for $\gamma\in\mathrm{Gal}(K_\infty/K)$.
\end{itemize}
\end{intro-conj1}

\begin{rem}
As formulated in \cite[Conj.~B]{perrin-riou-heegner}, the second equality of characteristic ideals in (ii) includes the factor 
$c_\pi\cdot (\#\mathcal{O}_K^\times)/2$, where $c_\pi\in\mathbf{Z}_{>0}$ is the Manin constant associated to $\pi$. However, $\mathcal{O}_K^\times=\{\pm 1\}$ by our hypothesis (\ref{assu:disc}), and $c_\pi$ is a $p$-adic unit by \cite[Cor.~3.1]{mazur-isogenies} and our hypothesis that $p\nmid N$.
\end{rem}

Mazur's conjecture on the non-triviality of $\kappa_\infty$ was first proved by Cornut--Vatsal \cite{cornut-vatsal}. Building on this, and adapting to the anticyclotomic setting the 
Kolyvagin system machinery of Mazur--Rubin \cite{mazur-rubin-book}, Howard \cite{howard-kolyvagin, howard-gl2-type} (extending earlier results by Bertolini \cite{bertolini}) reduced the proof of Conjecture~\ref{conj:HPMC} to the proof of the divisibility
\begin{equation}\label{eq:cong-div}
\mathrm{Char}_\Lambda(M_{\infty})\overset{?}\subset
\mathrm{Char}_\Lambda\bigl(S_\infty/\Lambda\kappa_\infty\bigr).
\end{equation}

More recently, the first case of this divisibility, and therefore of  Conjecture~\ref{conj:HPMC},  were obtained in \cite[Thm.~1.2]{wan-heegner} and \cite[Thm.~3.4]{castella-beilinson-flach}. The new ingredient in these works was Xin~Wan's divisibility in the Iwasawa--Greenberg main conjecture for certain Rankin--Selberg $p$-adic $L$-functions \cite{wan-rankin-selberg}, which in combination with the reciprocity law for Heegner points \cite{cas-hsieh1} yields a proof of the divisibility (\ref{eq:cong-div}). Unfortunately, the method in these works 
does not seem well suited to treat the case $N^-=1$ (i.e., the ``classical'' Heegner hypothesis), and for technical reasons they require the assumptions that $N$ is squarefree and that $p$ splits in $K$. 

In this paper, we give a proof of Conjecture~\ref{conj:HPMC} dispensing with the use of the deep results of \cite{wan-rankin-selberg} and allowing for the cases $N^-=1$, $N$ having square factors, and $p$ being inert in $K$.

\subsection{Statement of the main results} 

Let
\[
\overline{\rho}:G_{\mathbf{Q}}:=\mathrm{Gal}(\overline{\mathbf{Q}}/\mathbf{Q})\rightarrow\mathrm{Aut}_{\mathbf{F}_p}(E[p])
\]
be the Galois representation afforded by the $p$-torsion of $E$. Similarly as in  \cite{wei-zhang-mazur-tate}, we consider the following set of hypotheses on the triple $(E,p,K)$:

\begin{assumption-spade}
\label{assu:heart}
Let $\mathrm{Ram}(\overline{\rho})$ denote the set of primes $\ell\Vert N$ such that the $G_{\mathbf{Q}}$-module 
$E[p]$ is ramified at $\ell$. Then:
	\begin{itemize}
		\item[(i)] $\mathrm{Ram}(\overline{\rho})$ contains all primes $\ell\Vert N^+$,
		\item[(ii)] $\mathrm{Ram}(\overline{\rho})$ contains all primes  $\ell\vert N^-$ with $\ell\equiv\pm 1 \pmod{p}$,
		\item[(iii)] If $N$ is not squarefree, then either $\mathrm{Ram}(\overline{\rho})$ contains a prime $\ell\vert N^-$ or there are at least two primes $\ell\Vert N^+$.
	\end{itemize}
\end{assumption-spade}

Following the terminology introduced in \cite{mazur-IMC-for-E}, we say that  the prime $p$ is \emph{non-anomalous} if $p\nmid\vert\widetilde{E}(\mathbf{F}_w)\vert$ for all primes $w\vert p$ of $K$, where $\mathbf{F}_w$ is the residue field of $w$. 

Our main result towards Conjecture~\ref{conj:HPMC} is the following.

\begin{ThmA}\label{thm:main-intro}
Let $p>3$ be a prime where $E$ has good ordinary reduction, and let $K$ be an imaginary quadratic field satisfying $\mathrm{(\ref{assu:gen_heeg})}$ and $\mathrm{(\ref{assu:disc})}$. Assume that: 
\begin{itemize}
\item Hypothesis~$\spadesuit$ holds for (E,p,K),
\item $\overline{\rho}$ is surjective, 
\item $p$ is non-anomalous. 
\end{itemize} 
Then the Heegner point main conjecture holds.
\end{ThmA}

As a consequence of this result, we also obtain new cases of the Iwasawa--Greenberg main conjecture for certain Rankin--Selberg $p$-adic $L$-functions. Let $f\in S_2(\Gamma_0(N))$ be the newform associated with $E$. Assuming that 
\begin{equation}\label{assu:spl}
\textrm{$p\mathcal{O}_K=\mathfrak{p}\overline{\mathfrak{p}}$ splits in $K$}\tag{spl}
\end{equation}
and that $N^-=1$, Bertolini--Darmon--Prasanna  \cite{bertolini-darmon-prasanna-duke} constructed a $p$-adic $L$-function
\[ 
\mathscr{L}_{\mathfrak{p}}^{\bdp}\in\Lambda^{\mathrm{ur}}:=\Lambda\hat{\otimes}_{\mathbf{Z}_p}\mathbf{Z}^{\mathrm{ur}}_p
\]
with the property that $(\mathscr{L}_{\mathfrak{p}}^{\bdp})^2$ interpolates certain central critical $L$-values for the Rankin--Selberg convolution of $f$ with theta series attached to $K$ of weight $\ell\geqslant 3$, where $\mathbf{Z}^{\mathrm{ur}}_p$ is the completion of the ring of integers of the maximal unramified extension of $\mathbf{Q}_p$. The construction of $\mathscr{L}_{\mathfrak{p}}^{\bdp}$ was extended by Brooks \cite{brooks} to the case $N^-\neq 1$, and its corresponding interpolation property was deduced from calculations in \cite{prasanna} in the case where $N$ is squarefree. 

The Iwasawa--Greenberg main conjecture \cite{greenberg-motives} in this case predict that the square of $\mathscr{L}_{\mathfrak{p}}^{\bdp}$ generates the characteristic ideal of the Pontryagin dual of a Selmer group
\[
\mathrm{Sel}^{}_{\emptyset,0}(K,\mathbf{W})\subset \varinjlim_n\mathrm{H}^1(K_n,E[p^\infty])
\]
differing from $\varinjlim_n\mathrm{Sel}_{p^\infty}(E/K_n)$ in its defining local conditions at the primes above $p$.

\begin{conj}[{\bf Iwasawa--Greenberg main conjecture for $\mathscr{L}_{\mathfrak{p}}^{\bdp}$}]\label{conj:BDP}
Suppose $K$ satisfies $\mathrm{(\ref{assu:disc})}$, $\mathrm{(\ref{assu:gen_heeg})}$, and $\mathrm{(\ref{assu:spl})}$. Then the Pontryagin dual $X^{}_{\emptyset,0}$ of $\mathrm{Sel}^{}_{\emptyset,0}(K,\mathbf{W})$ is $\Lambda$-torsion, and 
\[
\mathrm{Char}_\Lambda(X^{}_{\emptyset,0})=(\mathscr{L}_{\mathfrak{p}}^{\mathrm{\bdp}})^2
\]
as ideals in $\Lambda^{\mathrm{ur}}$.
\end{conj}

In $\S\ref{sec:p-adicL}$ we extend the explicit reciprocity law of \cite{cas-hsieh1} (for weight $2$ forms) to the case $N^-\neq 1$, and use it to establish  the equivalence between  Conjectures~\ref{conj:HPMC} and \ref{conj:BDP}. (Such extension of \cite[Thm.~5.7]{cas-hsieh1} was used in the aforementioned works \cite{castella-beilinson-flach}, \cite{wan-heegner}, but the details were missing in the literature.) Together with Theorem~A we thus obtain the following. 

\begin{ThmB}\label{thm:main-BDP}
Let $p>3$ be a prime where $E$ has good ordinary reduction, 
and let $K$ be an imaginary quadratic field satisfying $\mathrm{(\ref{assu:gen_heeg})}$, $\mathrm{(\ref{assu:disc})}$, and $\mathrm{(\ref{assu:spl})}$. 
Assume that: 
\begin{itemize}
\item Hypothesis~$\spadesuit$ holds for (E,p,K),
\item $\overline{\rho}$ is surjective, 
\item $p$ is non-anomalous. 
\end{itemize} 
Then the Iwasawa--Greenberg main conjecture for $\mathscr{L}_{\mathfrak{p}}^{\bdp}$  holds.
\end{ThmB}

We conclude this subsection by noting another consequence of Theorem~A, which underlies the structure of its proof. 
As first observed in \cite{wan-heegner}, Perrin-Riou's Heegner point main conjecture implies a corresponding $p$-converse to the theorem of Gross--Zagier and Kolyvagin in the spirit of Skinner's work \cite{skinner-conv}: if $\mathrm{Sel}_{p^\infty}(E/K)$ has $\mathbf{Z}_p$-corank~one, then $\mathrm{ord}_{s=1}L(E/K,s)=1$. Indeed, the implication follows easily from Mazur's control theorem. In \cite{wei-zhang-mazur-tate}, this $p$-converse is deduced from the proof of Kolyvagin's conjecture in \emph{op.cit.} together with Kolyvagin's theorem \cite[Thm.~4]{kolyvagin-selmer} on the structure of $\mathrm{Sel}_{p^\infty}(E/K)$ (see \cite[Thm.~1.3]{wei-zhang-mazur-tate}). As a consequence of Theorem~A, the above $p$-converse can be deduced from W.~Zhang's proof of Kolyvagin's conjecture without the need to appeal to \cite{kolyvagin-selmer}. 

\subsection{Outline of the proofs}

As mentioned above, Howard's results towards Conjecture~\ref{conj:HPMC} were based on an adaptation to the anticyclotomic setting of the Kolyvagin system machinery of Mazur--Rubin \cite{mazur-rubin-book}, which provides upper bounds on the size of Selmer groups. As already observed by Kolyvagin \cite{kolyvagin-selmer}, the upper bound  provided by this machinery can be shown to be sharp under a certain non-vanishing hypothesis;  in the framework of \cite{mazur-rubin-book}, this corresponds to the Kolyvagin system being \emph{primitive}, see Definitions~4.5.5 and 5.3.9 in \cite{mazur-rubin-book}. 

Motivated by the ingenious Euler system argument introduced by  Bertolini--Darmon in \cite{bertolini-darmon-imc-2005}, Howard developed a theory of \emph{bipartite Euler systems} \cite{howard-bipartite}, which provides an alternative way to obtain upper bounds on Selmer groups without the need to apply Kolyvagin derivatives. Moreover, Howard also proved a criterion for his theory to yield a proof of the equality (rather than just one of the divisibilities) in a corresponding Iwasawa main conjecture. 

In a sense that will be made precise in $\S\ref{sec:bipartite}$, Howard's criterion for equality can be interpreted as the condition that the given bipartite Euler system is ``$\Lambda$-primitive''. On the other hand, digging into the proof of some of the main results in \cite{wei-zhang-mazur-tate}, we show that the constructions of Bertolini--Darmon \cite{bertolini-darmon-imc-2005} (as refined by Pollack--Weston \cite{pw-mu} and Chida--Hsieh \cite{chida-hsieh-main-conj}) yield a bipartite Euler system that is ``primitive''. Thus, by showing the implication 
\[
\textrm{primitivity$\quad\Longrightarrow\quad\Lambda$-primitivity} 
\]
for bipartite Euler systems, we arrive at the proof of Theorem~A. The proof of Theorem~B then follows from the equivalence between Conjectures~\ref{conj:HPMC} and \ref{conj:BDP} established in $\S\ref{sec:equiv-IMC}$.


\subsection{Acknowledgements} During the preparation of this work, F.C. was partially supported by the National Science Foundation through grants DMS-1801385 and DMS-1946136; C.K. was partially supported by a KIAS Individual Grant (SP054102) via the Center for Mathematical Challenges at Korea Institute for Advanced Study and by the Basic Science Research Program through the National Research Foundation of Korea (NRF-2018R1C1B6007009). It is a pleasure to thank Rob Pollack and Wei Zhang for their encouragement, Chris Skinner and Murilo Zanarella for several fruitful discussions, and the anonymous referee for a number of inquiries that led to significant improvements in the exposition of our results.

\section{Selmer groups}\label{sec:Sel}

Fix a prime $p>3$ and an embedding $\imath_p:\overline{\bQ}\hookrightarrow\overline{\bQ}_p$, where we let $\overline{\bQ}$ be the algebraic closure of $\mathbf{Q}$ in $\mathbf{C}$. Let $f=\sum_{n=1}^\infty a_nq^n\in S_2(\Gamma_0(N))$ be a newform with $p\nmid N$. Let $F=\mathbf{Q}(\{a_n\colon n\geqslant 1\})$ 
be the number field generated  by the Fourier coefficients of $f$, and let $\mathscr{O}$ be the ring of integers of $F$. We assume throughout that $f$ is ordinary at the prime $\pp$ of $\mathscr{O}$ above $p$ induced by $\imath_p$, i.e., $v_\pp(a_p)=0$.
 
Let $A_f$ be the $\mathrm{GL}_2$-type abelian variety over $\mathbf{Q}$ (unique up to isogeny) attached to $f$. Let $\mathscr{O}_\pp$ be the completion of $\mathscr{O}$ at $\pp$, and let
\[
T:=\varprojlim_j A_f[\pp^j] 
\]
be the $\pp$-adic Tate module of $A_f$, which is free of rank two over $\mathscr{O}_\pp$. Denote by $F_{\pp}$ the fraction field of $\mathscr{O}_\pp$, and set
\[
V:=T\otimes_{\mathscr{O}_\pp}F_\pp,\quad W:=V/T\simeq A_f[\pp^\infty].
\]

Let $K$ be an imaginary quadratic field of discriminant $D_K<0$ with $(D_K,N)=1$, and such that hypotheses (\ref{assu:disc}) and (\ref{assu:gen_heeg}) in the introduction hold. 
Let $K_\infty$ be the anticylotomic $\bZ_p$-extension of $K$, and let
\[ 
\Lambda=\mathscr{O}_\pp\llbracket \mathrm{Gal}(K_\infty/K )\rrbracket
\] 
be the anticylotomic Iwasawa algebra.

Fix a finite set $\Sigma$ of places of $K$ containing $\infty$ and the primes dividing $Np$, and let $K_{\Sigma}$ be the maximal extension of $K$ in $\overline{\mathbf{Q}}$ unramified outside $\Sigma$. Following \cite{mazur-rubin-book}, given a Selmer structure $\mathcal{F}$ on a $\mathrm{Gal}(K_\Sigma/K)$-module $M$, i.e., a collection of submodules $\mathrm{H}^1_{\mathcal{F}}(K_w, M)\subset\mathrm{H}^1(K_w, M)$ indexed by $w\in\Sigma$, we define the associated Selmer group 
by
\[
\mathrm{Sel}_{\mathcal{F}} (K, M) := \mathrm{ker} \biggl\{ \mathrm{H}^1(K_{\Sigma}/K, M) \rightarrow \prod_{w\in\Sigma} \dfrac{\mathrm{H}^1(K_w, M)}{\mathrm{H}^1_{\mathcal{F}}(K_w, M)} \biggr\}.
\]

\subsection{$p$-adic Selmer groups} 
Recall that if $M$ is a $G_K$-module and $L/K$ is a finite Galois extension, the induced representation 
\[
\mathrm{Ind}_{L/K}M := \lbrace f : G_K \to M : f(\sigma x) = f(x)^{\sigma} \textrm{ for all } x \in G_K, \sigma \in G_L \rbrace
\]
is equipped with commuting actions of $G_K$ and $\mathrm{Gal}(L/K)$. Consider the modules   
\begin{equation}\label{eq:TW}
{\displaystyle \mathbf{T} := \varprojlim_n \left( \mathrm{Ind}_{K_n/K}T\right)}, 
\quad
{\displaystyle \mathbf{W} := \varinjlim_n \left( \mathrm{Ind}_{K_n/K}W\right)} \simeq \mathrm{Hom}(\mathbf{T}, \mu_{p^\infty}), 
\end{equation}
where the limits are with respect to the corestriction and restriction maps, respectively,  and the isomorphism is given by the perfect $G_K$-equivariant pairing $\mathbf{T}\times\mathbf{W}\rightarrow\mu_{p^\infty}$ induced by the Weil pairing $T\times W\rightarrow\mu_{p^\infty}$ (see \cite[Prop.~2.2.4]{howard-kolyvagin}). Note that 
\[
\mathbf{T}\simeq T\otimes_{\mathscr{O}_\pp}\Lambda,
\]
where $G_K$ acts diagonally on the right-hand side, with the $G_K$-action on $\Lambda$ given by the inverse of the tautological character $G_K\twoheadrightarrow\mathrm{Gal}(K_\infty/K)\hookrightarrow\Lambda^\times$. We now describe certain Selmer groups for the modules $(\ref{eq:TW})$, whose $G_K$-action factors through $\mathrm{Gal}(K_\Sigma/K)$. 
 
Let $w$ be a prime of $K$ above $p$, and let $G_{K_w}\subset G_K$ be a decomposition group at $w$. Since $f$ is assumed to be ordinarity at $p$, there is a one-dimensional $G_{K_w}$-stable subspace $\mathrm{Fil}_w^+(V)\subset V$ such that the $G_{K_w}$-action on the quotient $V/\mathrm{Fil}_w^+V$ is unramified. Set
\[
\mathrm{Fil}_w^+(T) := T \cap \mathrm{Fil}_w^+(V),\quad
\mathrm{Fil}_w^+(W) :=  \mathrm{Fil}_w^+(V) / \mathrm{Fil}_w^+(T),
\]
and define the submodules $\mathrm{Fil}_w^+(\mathbf{T})\subset\mathbf{T}$ and $\mathrm{Fil}_w^+(\mathbf{W})\subset\mathbf{W}$ by
\[ 
\mathrm{Fil}_w^+(\mathbf{T}) := \varprojlim_n\left(\mathrm{Ind}_{K_n/K}\mathrm{Fil}_w^+(T)\right),\quad 
\mathrm{Fil}_w^+(\mathbf{W}) := \varinjlim_n\left(\mathrm{Ind}_{K_n/K}\mathrm{Fil}_w^+(W)\right).
\]
Following \cite[\S{3.2}]{howard-gl2-type}, we define the \emph{ordinary} Selmer structure $\mathcal{F}_{\textrm{ord}}$ on $\mathbf{T}$ by
\[
\mathrm{H}^1_{\mathcal{F}_{\textrm{ord}}}(K_w,\mathbf{T})=\left\{
\begin{array}{ll}
\mathrm{im}\{\mathrm{H}^1(K_w,\mathrm{Fil}_w^+(\mathbf{T}))\rightarrow\mathrm{H}^1(K_w,\mathbf{T})\}&\textrm{if $w\mid p$,}\\
\mathrm{H}^1(K_w,\mathbf{T})&\textrm{otherwise},
\end{array}
\right.
\]
and let $\mathrm{H}^1_{\mathcal{F}_{\mathrm{ord}}}(K_w,\mathbf{W})$ be the orthogonal complement of $\mathrm{H}^1_{\mathcal{F}_{\mathrm{ord}}}(K_w,\mathbf{T})$ under local Tate duality, so that
\[
\mathrm{H}^1_{\mathcal{F}_{\textrm{ord}}}(K_w,\mathbf{W})=
\begin{cases}
\mathrm{im}\{\mathrm{H}^1(K_w,\mathrm{Fil}_w^+(\mathbf{W}))\rightarrow\mathrm{H}^1(K_w,\mathbf{W})\}&\textrm{if $w\mid p$,}\\
0&\textrm{otherwise.}
\end{cases}
\]

We denote by $\mathrm{Sel}(K,\mathbf{T})$ and $\mathrm{Sel}(K,\mathbf{W})$ the Selmer groups defined by the Selmer structure $\mathcal{F}_{\mathrm{ord}}$. Shapiro's lemma gives canonical isomorphisms
\[
\mathrm{H}^1(K,\mathbf{T})\simeq\varprojlim_n\mathrm{H}^1(K_n,T),\quad
\mathrm{H}^1(K,\mathbf{W})\simeq\varinjlim_n\mathrm{H}^1(K_n,W),
\]
and as is well-known there are $\Lambda$-module pseudo-isomorphisms
\[
\mathrm{Sel}(K,\mathbf{T})\sim\varprojlim_n S_\pp(A_f/K_n),\quad
\mathrm{Sel}(K,\mathbf{W})\sim\varinjlim_n\mathrm{Sel}_{\pp^\infty}(A_f/K_n),
\]
where $S_\pp(A_f/L)$ and $\mathrm{Sel}_{\pp^\infty}(A_f/L)$ are the Selmer groups fitting into the exact sequences
\[
0\rightarrow A_f(L)\otimes\Phi_\pp/\mathscr{O}_\pp\rightarrow\mathrm{Sel}_{\pp^\infty}(A_f/L)\rightarrow\sha(A_f/L)[\pp^\infty]\rightarrow 0,
\]
\[
0\rightarrow A_f(L)\otimes\mathscr{O}_\pp\rightarrow S_\pp(A_f/L)\rightarrow\varprojlim_j\sha(A_f/L)[\pp^j]\rightarrow 0
\] 
(see e.g. \cite{coates-greenberg}).

\subsection{Residual Selmer groups}\label{subsec:residual}

Following \cite{bertolini-darmon-imc-2005}, we say that a prime $q$ is \emph{admissible} if it satisfies the following properties: 
\begin{itemize}
\item{} $q\nmid ND_Kp$, 
\item{} $q$ is inert in $K$, 
\item{} $q\not\equiv\pm{1}\pmod{p}$, 
\end{itemize}
and the ``admissibility index'' $M'(q):=v_\pp((q+1)^2-a_q^2)$ is strictly positive. 

We denote by $\mathcal{L}'$ the set of admissible primes and by $\mathcal{N}'$ the set of squarefree products of distinct primes $q\in\mathcal{L}'$.  
For $m\in\mathcal{N}'$ we define the admissibility index
\[
M'(m):=
\begin{cases}
\min\{M'(q)\colon q\mid m\}&\textrm{if $m>1$},\\
\infty&\textrm{if $m=1$,}
\end{cases}
\]
and say that $m$ is $j$-admissible if $M'(m)\geqslant j$. Let $\mathcal{N}'_j$ be the set of $j$-admissible integers $m\in\mathcal{N}'$, and let $\mathcal{N}'^{,\pm}$ (resp. $\mathcal{N}_j'^{,\pm}$) be the set of $m\in\mathcal{N}'$ (resp. $m\in\mathcal{N}_j'$) with $(-1)^{\nu(m)}=\pm{1}$, where $\nu(m)$ is the number of prime factors of $m$.

Given $j>0$ and $m\in\mathcal{N}_j'$ (which in our applications will be taken to be in $\mathcal{N}_j'^{,+}$), we now define ``$N^-m$-ordinary'' Selmer groups for the modules
\begin{equation}\label{eq:TW-j}
{\displaystyle \mathbf{T}_j := \varprojlim_n \mathrm{Ind}_{K_n/K}(T/\pp^jT)}, 
\quad
{\displaystyle \mathbf{W}_j := \varinjlim_n \mathrm{Ind}_{K_n/K}(A_f[\pp^j])}\nonumber
\end{equation}
(\emph{cf.} \cite[\S{3.1}]{howard-bipartite}, \cite[\S{1.2}]{chida-hsieh-main-conj}). Importantly, these Selmer groups will depend on $N^-m$ and the reduction of $f$ modulo $\pp^j$, but not on $f$ itself.

Let $w$ be a prime of $K$ above $p$, and define 
$\mathrm{Fil}^+_w(A_f[\pp^j])$ to be the kernel of the reduction map
\[
A_f[\pp^j]\rightarrow\tilde{A}_f[\pp^j],
\]
where $\tilde{A}_f$ is the reduction of $A_f$ modulo $w$. Set 
\[
\mathrm{Fil}_w^+(\mathbf{W}_j) = \varinjlim_n \mathrm{Ind}_{K_n/K}\mathrm{Fil}_w^+(A_f[\pp^j])
\] 
and define the \emph{ordinary} condition $\mathrm{H}^1_{\mathrm{ord}}(K_w,\mathbf{W}_j)\subset\mathrm{H}^1(K_w,\mathbf{W}_j)$ by
\begin{equation}\label{eq:ord}
\mathrm{H}^1_{\mathrm{ord}}(K_w,\mathbf{W}_j):=\mathrm{im}\bigl\{\mathrm{H}^1(K_w,\mathrm{Fil}_w^+(\mathbf{W}_j))\rightarrow\mathrm{H}^1(K_w,\mathbf{W}_j)\bigr\}.
\end{equation}

Next let $w$ be a prime of $K$ above a prime $\ell\mid N^-m$. If $\ell\mid N^-$, then $A_f$ acquires purely toric reduction over $K_w=\mathbf{Q}_{\ell^2}$, and by the theory of $\ell$-adic uniformization of Tate and Morikawa there is a unique rank one $\mathscr{O}_\pp$-submodule $\mathrm{Fil}_w^+(T)\subset T$ on which $G_{K_w}$ acts by the cyclotomic character. Letting $\mathrm{Fil}_w^+(A_f[\pp^j])$ be the natural image of $\mathrm{Fil}_w^+(T)$ in  $A_f[\pp^j]$, we define the ordinary condition at $w$ as in $(\ref{eq:ord})$. On the other hand, if $\ell\mid m$, then the Galois module 
$A_f[\pp^j]$ is unramified at $w$ and the action of a Frobenius element at $w$ is semi-simple, yielding a decomposition 
\begin{equation}\label{eq:dec}
A_f[\pp^j]\simeq(\mathscr{O}/\pp^j)(1)\oplus(\mathscr{O}/\pp^j)
\end{equation}
as $G_{K_w}$-modules. Letting $\mathrm{Fil}^+_{w}(A_{f}[\pp^j])\subset A_f[\pp^j]$ be the direct summand corresponding to the first factor    in the decomposition $(\ref{eq:dec})$, we define the ordinary submodule $\mathrm{H}^1_{\mathrm{ord}}(K_w,\mathbf{W}_j)\subset\mathrm{H}^1(K_w,\mathbf{W}_j)$ by the same recipe $(\ref{eq:ord})$.

Following \cite[Def.~2.8]{bertolini-darmon-imc-2005}, we define the ``$N^-m$-ordinary'' Selmer group $\mathrm{Sel}_{N^-m}(K,\mathbf{W}_j)$ to be the Selmer group defined by
\begin{itemize} 
\item the ordinary local condition $\mathrm{H}^1_{\mathrm{ord}}(K_w,\mathbf{W}_j)$ at the primes $w\mid pN^-m$, 
\item the unramified local condition 
\[
\mathrm{H}^1_{\mathrm{unr}}(K_w,\mathbf{W}_j):=\mathrm{ker}\bigr\{\mathrm{H}^1(K_w,\mathbf{W}_j)\rightarrow\mathrm{H}^1(K_w^{\mathrm{unr}},\mathbf{W}_j)\bigr\}
\] 
at all the other primes. 
\end{itemize}

Since $T/\pp^jT\simeq A_f[\pp^j]$, the ordinary submodules $\mathrm{H}^1_{\mathrm{ord}}(K_w,\mathbf{T}_j)\subset\mathrm{H}^1(K_w,\mathbf{T}_j)$ for $w\mid pN^-m$, and the corresponding Selmer group $\mathrm{Sel}_{N^-m}(K,\mathbf{T}_j)$ can be defined in the same manner.

In the following, abusing notation, given $q\in\mathcal{L}'$ we shall denote by $K_q$ the completion of $K$ at the unique prime above $q$.

\begin{lem}\label{lem:rank1}
For any $q\in\mathcal{L}_j$, the modules
\[
\mathrm{H}^1_{\mathrm{ord}}(K_q,\mathbf{T}_j),\quad\mathrm{H}^1_{\mathrm{unr}}(K_q,\mathbf{T}_j)
\]
are free of rank one over $\Lambda/\pp^j\Lambda$.
\end{lem}

\begin{proof}
Since the primes $q\in\mathcal{L}_j$ are inert in $K$, they split completely in $K_\infty/K$, and so Shapiro's lemma gives an isomorphism
\[
\mathrm{H}^1(K_q,\mathbf{T}_j)\simeq\varprojlim_n\bigoplus_{w\vert q}\mathrm{H}^1(K_{n,w},A_f[\pp^j])\simeq\mathrm{H}^1(K_q,A_f[\pp^j])\otimes\Lambda,
\]
where $w$ runs over the primes of $K_n$ above $q$. By \cite[Lem.~2.2.1]{howard-bipartite}, the result follows.
\end{proof}

The next result compares the $p$-adic Selmer groups 
for the ordinary Selmer structure $\mathcal{F}_{\mathrm{ord}}$ defined above and the corresponding $N^-$-ordinary Selmer groups.

\begin{lem}\label{lem:ord-min}
Assume that $A_f[\pp]$ is irreducible as a $G_\mathbf{Q}$-module. Then 
\[
\mathrm{Sel}(K,\mathbf{W})\simeq\varinjlim_j\mathrm{Sel}_{N^-}(K,\mathbf{W}_j),\quad
\mathrm{Sel}(K,\mathbf{T})\simeq\varprojlim_j\mathrm{Sel}_{N^-}(K,\mathbf{T}_j).
\]
\end{lem}

\begin{proof}
The second identification follows immediately from the first. For the latter, as shown in the proof of \cite[Prop.~3.6]{pw-mu}, $\varinjlim_j\mathrm{Sel}_{N^-}(K,\mathbf{W}_j)$ is contained in $\mathrm{Sel}(K,\mathbf{W})$ with finite index. Since by \cite[Prop.~3.12]{hatley-lei} the module $\mathrm{Sel}(K,\mathbf{W})$ has no proper finite index submodules (note that this result does not require this module to be $\Lambda$-torsion), the first isomorphism follows. 
\end{proof}

\section{Proof of Theorem~A}\label{sec:bipartite}

We keep the setting and notations introduced in Section~\ref{sec:Sel}. Let $\mathbf{F}=\mathscr{O}/\pp$ be the residue field of $\pp$, and let 
\[
\overline{\rho}:G_\bQ\rightarrow\mathrm{Aut}_{\mathbf{F}}(A_f[\pp])\simeq\mathrm{GL}_2(\mathbf{F})
\] 
be the Galois representation on the $\pp$-torsion of $A_f$. Let $\mathscr{O}_0\subset\mathscr{O}$ be the order generated over $\bZ$ by the Fourier coefficients of $f$, and set $\wp_0:=\wp\cap\mathscr{O}$ and $\mathbf{F}_0:=\mathscr{O}_0/\wp_0$.  
Note that $\overline{\rho}$ arises as the extension of scalars of a representation $\overline{\rho}_0$ defined over $\mathbf{F}_0$.

As in \cite{wei-zhang-mazur-tate}, we consider the following conditions on the triple $(f,\pp,K)$.

\begin{assumption-heart}
\label{assu:heart}
Let $\mathrm{Ram}(\overline{\rho})$ denote the set of primes $\ell\Vert N$ such that the $G_{\mathbf{Q}}$-module 
$A_f[\wp]$ is ramified at $\ell$. Then:
	\begin{itemize}
		\item[(i)] $\mathrm{Ram}(\overline{\rho})$ contains all primes $\ell\Vert N^+$. 
		\item[(ii)] $\mathrm{Ram}(\overline{\rho})$ contains all primes  $\ell\vert N^-$ with $\ell\equiv\pm 1 \pmod{p}$. 
		\item[(iii)] If $N$ is not squarefree, then either $\mathrm{Ram}(\overline{\rho})$ contains a prime $\ell\vert N^-$ or there are at least two primes $\ell\Vert N^+$.
		\item[(iv)] For all primes $\ell$ such that $\ell^2\vert N^+$ we have $\mathrm{H}^1(\mathbf{Q}_\ell,A_f[\wp])=\mathrm{H}^0(\mathbf{Q}_\ell,A_f[\wp])=\{0\}$.
	\end{itemize}
\end{assumption-heart}

\begin{rem}\label{rem:5.1}
When $\mathscr{O}=\mathbf{Z}$, i.e., for $f$ corresponding to an elliptic curve $E/\mathbf{Q}$,  Hypothesis~$\heartsuit$ for $(E,p,K)$ reduces to Hypothesis~$\spadesuit$ in the introduction. See \cite[Lem.~5.1(2)]{wei-zhang-mazur-tate}.
\end{rem}

Following Mazur's terminology in \cite{mazur-IMC-for-E}, we say that $p$ is \emph{non-anomalous} if 
\[
\begin{cases}
a_p\not\equiv 1\;(\mathrm{mod}\;\pp)&\textrm{if $p$ splits in $K$,}\\
a_p^2\not\equiv 1\;(\mathrm{mod}\;\pp)&\textrm{if $p$ is inert in $K$.}
\end{cases}
\]
In this section we prove the following result, which in the case where $f$ has rational Fourier coefficients recovers Theorem~A in the introduction. Let
\[
\mathcal{S}=\mathrm{Sel}(K,\mathbf{T}),\quad\mathcal{X}=\mathrm{Sel}(K,\mathbf{W})^\vee,
\]
where $M^\vee=\mathrm{Hom}_{\mathbf{Z}_p}(M,\mathbf{Q}_p/\mathbf{Z}_p)$ denotes the Pontryagin dual of a module $M$.

\begin{thm}\label{thm:main}
Suppose $p\nmid 6N$ and $\pp$ is a prime of $\mathscr{O}$ above $p$ such that the following hold: 
\begin{itemize}
\item $f$ is ordinary at $\pp$,
\item Hypothesis~$\heartsuit$ holds for $(f,\pp,K)$,
\item $\overline{\rho}_0$ is surjective,
\item $p$ is non-anomalous. 
\end{itemize}
Then both $\mathcal{S}$ and $\mathcal{X}$ have $\Lambda$-rank one, and
\[
\mathrm{Char}_\Lambda(\mathcal{X}_{\mathrm{tors}})=\mathrm{Char}_\Lambda\bigl(\mathcal{S}/\Lambda\kappa_\infty\bigr)^2,
\]
where $\mathcal{X}_{\mathrm{tors}}$ denotes the $\Lambda$-torsion submodule of $\mathcal{X}$.
\end{thm} 

The proof of Theorem~\ref{thm:main} will occupy the rest of this section. Our argument is based on Howard's theory of bipartite Euler systems, 
with some ideas and results from Wei~Zhang's proof of Kolyvagin's conjecture \cite{wei-zhang-mazur-tate}. In the terminology of \emph{loc.cit.}, we crucially exploit the ``$m$-aspect'' of the system of Heegner classes given by level-raising at admissible primes, rather than the ``$n$-aspect'' given by tame derivatives at Kolyvagin primes.

As in \cite{pw-mu}, we say that the pair $(\overline{\rho},N^-)$ satisfies \emph{Condition~CR} if the following hold: 

\begin{itemize}
\item[(i)] $\overline{\rho}$ is ramified at every prime $\ell\vert N^-$ with $\ell\equiv\pm{1}\pmod{p}$,
\item[(ii)] $\overline{\rho}_0$ is surjective.
\end{itemize}

The basic construction for our argument is provided by the following result coming from the work Bertolini--Darmon \cite{bertolini-darmon-imc-2005}, and its refinements by Pollack--Weston \cite{pw-mu} and Chida--Hsieh \cite{chida-hsieh-main-conj}. (Note that the condition that $p$ is non-anomalous made in the aforementioned references---see also  \cite[Rem.~1.4]{KPW}---is no longer necessary thanks to recent advances on Ihara's lemma, \cite{manning-shotton}.)


\begin{thm}
\label{thm:bipartite}
Suppose $(\overline{\rho},N^-)$ satisfies Condition~CR. Then for every $j>0$ there is a pair of systems
\begin{align*}
\boldsymbol{\kappa}&=\{\kappa_j(m)\in\mathrm{Sel}_{N^-m}(K,\mathbf{T}_j)\colon m\in\mathcal{N}_j'^{,+}\},
\\
\boldsymbol{\lambda}&=\{\lambda_j(m)\in\Lambda/\pp^{j}\Lambda \colon m\in\mathcal{N}_j'^{,-}\},
\end{align*}
related by a system of ``explicit reciprocity laws'': 
\begin{itemize}
\item{} If $mq_1q_2\in\mathcal{N}_{j}'^{,+}$ with $q_1, q_2\in\mathcal{L}_j'$ distinct primes, then
\begin{equation}\label{ERL1}
\mathrm{loc}_{q_2}(\kappa_j(mq_1q_2))=\lambda_j(mq_1)\tag{1st}
\end{equation}
under a fixed isomorphism $\mathrm{H}_{\mathrm{ord}}^1(K_{q_1},\mathbf{T}_j)\simeq\Lambda/\pp^j\Lambda$ (see Lemma~\ref{lem:rank1});
\item{} If $mq\in\mathcal{N}_{j}'^{,-}$ with $q\in\mathcal{L}_j'$ prime, then
\begin{equation}\label{ERL2}
\mathrm{loc}_q(\kappa_j(m))=\lambda_j(mq)\tag{2nd}
\end{equation}
under a fixed isomorphisms $\mathrm{H}_{\mathrm{unr}}^1(K_{q},\mathbf{T}_j)\simeq\Lambda/\pp^j\Lambda$.
\end{itemize}
\end{thm}

\begin{proof}
We recall the construction of the systems $\boldsymbol{\kappa}$ and $\boldsymbol{\lambda}$, following the treatment in \cite{chida-hsieh-main-conj} with some modifications.   
Denote by $\alpha_p$ the $\pp$-adic unit root of $x^2-a_px+p$, and let $f_\alpha\in S_2(\Gamma_0(Np))$ be the $p$-stabilization of $f$ with $U_p$-eigenvalue $\alpha_p$. 
Fix $m\in\mathcal{N}'^{,+}$, let $B_m$ be the indefinite quaternion algebra over $\mathbf{Q}$ of discriminant $N^-m$, and consider the compact Shimura curve 
\[
X_m:=X_{pN^+,N^-m}
\] 
attached to an Eichler order $R_m\subset B_m$ of level $pN^+$ as defined in \cite[\S{4.2}]{jetchev-skinner-wan}. In particular, $X_{m}=X_0(Np)$ when $N^-m=1$). The curve $X_m$ has a canonical model over $\mathbf{Q}$, and its complex uniformization is given by
\begin{equation}\label{eq:complex-unif}
X_{m}(\mathbf{C})=B_m^\times\backslash\bigl(\mathfrak{H}^{\pm}\times\widehat{B}_m^\times/\widehat{R}_m^\times\bigr)\cup\{\textrm{cusps}\},
\quad\mathfrak{H}^\pm:=\mathbf{C}\smallsetminus\mathbf{R},
\end{equation}
where $\widehat{B}_m:=B_m\otimes_{\mathbf{Z}}\widehat{\mathbf{Z}}$ and $\widehat{R}_m:=R_m\otimes_{\mathbf{Z}}\widehat{\mathbf{Z}}$ are the profinite completions of $B_m$ and $R_m$. Fix an (optimal) embedding $\iota_K:K\hookrightarrow B_m$ such that 
\[
\iota_K(K)\cap R=\iota_K(\mathcal{O}_K),
\] 
where $\mathcal{O}_K$ is the ring of integers of $K$. In terms of the complex uniformization (\ref{eq:complex-unif}), the collection of \emph{Heegner points} on $X_{m}$ is defined as
\begin{equation}\label{eq:CM}
\mathrm{CM}(X_{m}):=\{[h,b]\in X_{m}(\mathbf{C})\colon b\in\widehat{B}_m^\times\}\simeq K^\times\backslash\widehat{B}_m^\times/\widehat{R}_m^\times,
\end{equation}
where $h$ is the unique fixed point of $\imath_K(K^\times)$ on $\mathfrak{H}^+$. We shall define Heegner points by specifying a representative element in $\widehat{B}_m^\times$ under the identification (\ref{eq:CM}). 

Let $J(X_{m})=\mathrm{Pic}^0(X_m)_{/\mathbf{Q}}$ be the Picard variety of $X_{m}$, and let $g_m$  be a mod $\pp^j$ level-raising of $f_\alpha$ of level $pNm$. The existence of $g_m$ follows from \cite[Thm.~4.3]{chida-hsieh-main-conj} (\emph{cf.} \cite[Thm.~5.15]{bertolini-darmon-imc-2005}), 
which also implies its uniqueness up to a $\pp$-adic unit. Moreover, letting $\mathbb{T}_m$ be the algebra of Hecke correspondences on $X_{m}$, by \cite[Cor.~4.4]{chida-hsieh-main-conj} there is an isomorphism
\begin{equation}\label{eq:4.4}
(\mathrm{Ta}_p(J(X_m))\otimes_{\mathbf{Z}_p}\mathscr{O}_\pp)/\mathcal{I}_{g_m}\simeq T/\pp^jT,
\end{equation}
where $\mathcal{I}_{g_m}$ is the kernel of the algebra homomorphism $\lambda_{g_m}:\mathbb{T}_{m}\rightarrow\mathscr{O}/\pp^j$ defined by $g_m$.

For each positive integer $n$, let $H_{p^n}$ be the ring class field of $K$ of conductor $p^n$, and let
\begin{equation}\label{eq:CM}
x_m(p^n)\in\mathrm{CM}(X_m)\cap X_m(H_{p^n})
\end{equation}
be the Heegner point defined by the element $\varsigma^{(n)}\in\widehat{B}^\times$ in \cite[(4.6)]{chida-hsieh-main-conj}. Choose an auxiliary prime $\ell_0$ such that $a_{\ell_0}-\ell_0-1\not\in\pp$ (note that the existence of such $\ell_0$ is guaranteed by the irreducibility of $\overline{\rho}$), and consider the map
\begin{equation}\label{eq:emb}
\begin{split}
\iota_{m}:X_{m}(H_{p^n})&\rightarrow J(X_{m})(H_{p^n})\otimes_{\mathbf{Z}}\mathscr{O}_\pp\\ 
x&\mapsto (T_{\ell_0}-\ell_0-1)[x]\otimes(a_{\ell_0}-\ell_0-1)^{-1},
\end{split}
\end{equation}
where $T_{\ell_0}$ is the $\ell_0$-th Hecke correspondence of $X_{m}$.  Let
\[
\mathrm{Kum}:J(X_m)(H_{p^n})\otimes_{\mathbf{Z}}\mathscr{O}_\pp\rightarrow\mathrm{H}^1(H_{p^n},\mathrm{Ta}_p(J(X_m))\otimes_{\mathbf{Z}_p}\mathscr{O}_\pp)
\]
be the Kummer map, and set $K_n:=H_{p^n}\cap K_\infty$. A standard calculation (see \cite[Lem.~4.6]{chida-hsieh-main-conj} and the reference \cite[Prop.~4.8]{longo-vigni-manuscripta} therein) shows that the classes defined by 
\[
\alpha_p^{-n}\sum_{\sigma\in\mathrm{Gal}(H_{p^n}/K_n)}\mathrm{Kum}(\iota_m(x_m(p^n))^\sigma)\;(\textrm{mod}\;\mathcal{I}_{g_m})
\]  
are compatible under corestriction, and hence under the isomorphism $(\ref{eq:4.4})$ they define a class
\[
\kappa_j(m)\in\varprojlim_n\mathrm{H}^1(K_n,T/\pp^jT)\simeq\mathrm{H}^1(K,\mathbf{T}_j),
\]
which by \cite[Prop.~4.7]{chida-hsieh-main-conj} lands in $\mathrm{Sel}_{N^-m}(K,\mathbf{T}_j)\subset\mathrm{H}^1(K,\mathbf{T}_j)$.

Now let $g_{mq}$ be a mod $\pp^j$ level-raising of $f_\alpha$ to level $pNmq$, viewed as an automorphic form on the Shimura set
\[
X_{mq}=B_{mq}^\times\backslash\widehat{B}_{mq}^\times/\widehat{R}^\times_{mq}
\]
attached to the definite quaternion algebra $B_{mq}$ over $\mathbf{Q}$ of discriminant $N^-mq$ with the Eichler order $R_{mq}\subset B_{mq}$ of level $pN^+$. 
Let $\mathcal{O}_{p^n}=\mathbf{Z}+p^n\mathcal{O}_K$ be the order of $K$ of conductor $p^n$, so that $\mathrm{Pic}(\mathcal{O}_{p^n})\simeq\mathrm{Gal}(H_{p^n}/K)$ under the reciprocity map $\mathrm{rec}_K:K^\times\backslash\widehat{K}^\times\rightarrow G_K^{\mathrm{ab}}$.  For a fixed embedding $K\hookrightarrow B_{mq}$, define the map
\[
x_{mp}(p^n):\mathrm{Pic}(\mathcal{O}_{p^n})=K^\times\backslash\widehat{K}^\times/\widehat{\mathcal{O}}_{p^n}^\times\rightarrow X_{mq}
\]
sending $K^\times a\widehat{\mathcal{O}}_{p^n}\mapsto[a\varsigma^{(n)}]$. Using that the mod $\pp^j$ eigenform $g_{mp}$ is a $U_p$-eigenvector with eigenvalue $\alpha_p$, one checks immediately that the natural image in $(\mathscr{O}/\pp^j)[\mathrm{Gal}(K_n/K)]$ of the element 
\[
\alpha_p^{-n}\sum_{\sigma\in\mathrm{Gal}(H_{p^n}/K)}g_{mp}(x_{mq}(p^n)(a))[\sigma]\in(\mathscr{O}/\pp^j)[\mathrm{Gal}(H_{p^n}/K)],
\]
where $\sigma=\mathrm{rec}_K(a)$, are compatible under the projections 
\[
(\mathscr{O}/\pp^j)[\mathrm{Gal}(K_{n}/K)]\rightarrow(\mathscr{O}/\pp^j)[\mathrm{Gal}(K_{n-1}/K)],
\]
hence defining an element
\[
\lambda_j(m)\in\varprojlim_n(\mathscr{O}/\pp^j)[\mathrm{Gal}(K_n/K)]\simeq\Lambda/\pp^j\Lambda.
\]

This defines the systems $\boldsymbol{\kappa}$ and $\boldsymbol{\lambda}$, and with these, equalities (\ref{ERL1}) and (\ref{ERL2}) in the statement of the Theorem are a reformulation of the first and second explicit reciprocity laws in \cite[Thm.~5.1]{chida-hsieh-main-conj} and \cite[Thm.~5.5]{chida-hsieh-main-conj}, respectively.
\end{proof}

In the terminology of \cite{howard-bipartite}, the pair $(\boldsymbol{\kappa},\boldsymbol{\lambda})$ defines a \emph{bipartite Euler system} (of odd type) for the triple $(A_f[\pp^j],\mathcal{F},\mathcal{L}'_j)$, where $\mathcal{F}$ is the Selmer structure defining the $N^-$-ordinary Selmer groups $\mathrm{Sel}_{N^-}(K,\mathbf{W}_j)$ and $\mathrm{Sel}_{N^-}(K,\mathbf{T}_j)$. Letting $\mathscr{X}$ be the graph with vertices $v=v(m)$ indexed by $m\in\mathcal{N}_j'$ and edges connecting $v(m)$ to $v(mq)$ whenever $q\in\mathcal{L}_j'$ and $mq\in\mathcal{N}_j'$, we shall use the interpretation of such systems as global sections of the  sheaf $\mathrm{ES}(\mathscr{X})$ on $\mathscr{X}$ introduced in \cite[\S{2.4}]{howard-bipartite}.

For varying $j$ the elements $\kappa_j(m)$ and $\lambda_j(m)$ are compatible under the natural maps 
\[
A_f[\pp^{j+1}]\rightarrow A_f[\pp^j],\quad
\Lambda/\pp^{j+1}\Lambda\rightarrow\Lambda/\pp^j\Lambda.
\]
Taking $m=1$ we thus obtain a distinguished element 
\begin{equation}\label{eq:dist}
\kappa_\infty:=\varprojlim_j\kappa_j(1)\in
\varprojlim_j\mathrm{Sel}_{N^-}(K,\mathbf{T}_j)
\end{equation}
using by the isomorphism in Lemma~\ref{lem:ord-min} we shall view in $\mathcal{S}=\mathrm{Sel}_{}(K,\mathbf{T})$.

We will prove Theorem~\ref{thm:main} by an application of the following result of Howard.

\begin{thm}[Howard]\label{thm:howard-bipartite}
Assume that the pair $(\overline{\rho},N^-)$ satisfies Condition~CR. Then both $\mathcal{S}$ and $\mathcal{X}$ have $\Lambda$-rank one, and  the following divisibility holds in $\Lambda$:
\begin{equation}\label{eq:div}
\mathrm{Char}_\Lambda(\mathcal{X}_{\mathrm{tors}})\supset\mathrm{Char}_\Lambda\bigl(\mathcal{S}/\Lambda\kappa_\infty\bigr)^2.
\end{equation}
Moreover, the divisibility in $(\ref{eq:div})$ is an equality if the following condition is satisfied: For any height one prime $\mathfrak{P}\subset\Lambda$, there exists $k=k(\mathfrak{P})$  such that for all $j\geqslant k$ the set
\begin{equation}\label{eq:prim}
\{\lambda_j(m)\in\Lambda/\pp^{j}\Lambda\;\colon\; m\in\mathcal{N}_j'^{,-}\}\nonumber
\end{equation}
contains an element with non-trivial image in $\Lambda/(\mathfrak{P},\pp^k)$.
\end{thm}

\begin{proof}
The element $\kappa_\infty$ is nonzero by the work of Cornut--Vatsal \cite{cornut-vatsal}, so the result follows from Lemma~\ref{lem:ord-min} and \cite[Thm.~3.2.3]{howard-bipartite}.
\end{proof}

In the following lemma, let $(R,\mathfrak{m}_R)$ be a principal Artinian local ring, let $T$ be a free $R$-module of rank $2$ equipped with a continuous action of $G_K$ as in \cite[\S{2.6}]{howard-bipartite}, let $\mathcal{F}$ be a Selmer structure on $T$,  
and let $\mathcal{L}'$ be a set of (admissible) primes of $K$ 
such that $(T,\mathcal{F},\mathcal{L}')$ satisfies Hypotheses~2.2.4 and 2.3.1 of \cite{howard-bipartite}. We refer the reader to \cite[Def.~2.2.8]{howard-bipartite} for the definition of the \emph{stub module} 
\[
\mathrm{Stub}(v)=\mathrm{Stub}_m\subset R
\] 
associated with the vertex $v$ of $\mathscr{X}$ indexed by $m$, and (as in [\emph{loc.cit.}, Def.~2.4.2]) say that $v$ is a \emph{core vertex} if $\mathrm{Stub}(v)=R$.


\begin{lem}\label{lem:core}
Let $s$ be the global section of $\mathrm{ES}(\mathscr{X})$ corresponding to a bipartite Euler system over $R$. Then there exists a constant $\delta=\delta(s)$ with 
$0\leqslant\delta\leqslant\mathrm{length}(R)$ such that $s(v)$ generates
$\mathfrak{m}_R^\delta\cdot\mathrm{Stub}(v)$ for every core vertex $v$ of $\mathscr{X}$. Moreover, $s$ is uniquely determined by its value at any core vertex.
\end{lem}

\begin{proof}
This is shown in the proof of \cite[Cor.~2.4.12]{howard-bipartite}.
\end{proof}

Now we return to our setting.

\begin{lem}\label{prop:prim}
Suppose $(\overline{\rho},N^-)$ satisfies Condition~CR. If the system $\boldsymbol{\lambda}$ has nonzero image in $\Lambda/\pp\Lambda$, then the criterion for equality 
in Theorem~\ref{thm:howard-bipartite} holds.
\end{lem}

\begin{proof}
Unless indicated otherwise, all the references in this proof are to \cite{howard-bipartite}. 
Denote by $\overline{\boldsymbol{\lambda}}$ the image of $\boldsymbol{\lambda}$ in $\Lambda/\pp\Lambda$. Denoting by $\overline{\boldsymbol{\kappa}}$ the reduction of $\boldsymbol{\kappa}$ modulo $\pp$, the pair $(\overline{\boldsymbol{\lambda}},\overline{\boldsymbol{\kappa}})$ defines a bipartite Euler system over $\mathbf{F}$; or equivalently, a global section $s$ of the corresponding Euler system sheaf $\mathrm{ES}(\mathscr{X})$.  
Since $\overline{\boldsymbol{\lambda}}\neq 0$ and by Corollary~2.4.9 there are core vertices corresponding to $m\in\mathcal{N}_j'^{,-}$ for any $j$, by Lemma~\ref{lem:core} above (noting that Hypothesis~2.2.4 holds by our running hypotheses, and Hypothesis~2.3.1 holds by \cite[Thm.~3.2]{bertolini-darmon-imc-2005}) it follows that $s(v)\neq 0$ for any core vertex of $\mathscr{X}$. Since $\mathbf{F}$ has length one, this shows that $\delta=0$ in Lemma~\ref{lem:core} above. Thus we conclude that for any $j>0$ the system
\[
\{\lambda_j(m)\in\Lambda/\pp^{j}\Lambda\;\colon\; m\in\mathcal{N}_j'^{,-}\}
\]
has nonzero image in $\Lambda/\pp\Lambda$, and so for any height one prime $\mathfrak{P}\subset\Lambda$ the criterion in Theorem~\ref{thm:howard-bipartite} is satisfied by taking $k=k(\mathfrak{P})=1$. 
\end{proof}

\begin{prop}\label{prop:wei}
Suppose the following hold:
\begin{itemize}
\item $(f,\wp,K)$ satisfies Hypothesis~$\heartsuit$, 
\item $\overline{\rho}_0$ is surjective, 
\item $p$ is non-anomalous. 
\end{itemize}
Then the system $\boldsymbol{\lambda}$ has nonzero image in $\Lambda/\pp\Lambda$.
\end{prop}

\begin{proof}
All the references in this proof are to \cite{wei-zhang-mazur-tate}. 
Let $\mathrm{Sel}_{\wp}(A_f/K)\subset\mathrm{H}^1(K,A_f[\wp])$ be the usual $\wp$-Selmer group, and set
\[
r=\mathrm{dim}_{\mathbf{F}}\mathrm{Sel}_\wp(A_f/K).
\] 
We need to show that for some $g$ obtained by level-raising $f$ at $m\in\mathcal{N}'^{,-}$, 
the $p$-adic $L$-function attached to $g$ over $K$ (as constructed in \cite{chida-hsieh-p-adic-L-functions} using the period denoted by $\Omega_g^{\mathrm{can}}\cdot\eta_{g,N^+,N^-m}$ in the notations of \cite[\S{6.2}]{wei-zhang-mazur-tate}, i.e., Gross's period) is invertible.  We will show this by induction on $r$. 

Since $K$ satisfies hypothesis (\ref{assu:gen_heeg}), as in Theorem~9.1 we may assume that $r$ odd.  
If $r=1$, the existence of $g$ is shown in Theorem~7.2 (where it is denoted by $g'$). Indeed, $g$ is obtained by mod $\pp$ level-raising $f$ at some $q\in\mathcal{L}'$, and the proof of Theorem~7.2 shows that 
\[
L^{\mathrm{alg}}(g/K)\not\equiv 0\pmod{\pp}
\]
(see bottom of p.~233). On the other hand, by the interpolation formula in \cite[Thm.~A]{chida-hsieh-p-adic-L-functions}, the image of  $\lambda_1(q)^2$ under the augmentation map $\Lambda/\pp^j\Lambda\rightarrow\mathscr{O}_\pp/\pp^j\mathscr{O}_\pp$ (corresponding to the evaluation at the trivial character $\mathds{1}$ of $\mathrm{Gal}(K_\infty/K)$) 
is given by
\[
e_p(g,\mathds{1})\cdot L^{\mathrm{alg}}(g/K)\pmod{\pp}
\]
up to a $p$-adic unit, where $e_p(g,\mathds{1})$ is a certain $p$-adic multiplier. Since $e_p(g,\mathds{1})\not\equiv 0\pmod{\pp}$ by the non-anomalous hypothesis on $p$, the result in the case $r=1$ follows. 

If $r\geqslant 3$, by the argument in the proof of Theorem~9.1 we can find a form $g_2$ of level $Nq_1q_2$, obtained by level-raising $f$ at two distinct admissible primes $q_1$ and $q_2$, with associated Selmer rank equal to $r-2$. By induction hypothesis, $g_2$ has a mod~$\pp$ level-raised form $g$ as desired, and therefore so does $f$.
\end{proof}

Now we have all the ingredients to prove Theorem~\ref{thm:main}.

\begin{proof}[Proof of Theorem~\ref{thm:main}]
By Theorem~\ref{thm:howard-bipartite} and Lemma~\ref{prop:prim}, it suffices to show that $\boldsymbol{\lambda}$ has nonzero image in $\Lambda/\pp\Lambda$, which under the hypotheses of Theorem~\ref{thm:main} has been shown in Proposition~\ref{prop:wei}, hence the result.
\end{proof}

\section{The $p$-adic $L$-function $\mathscr{L}_{\mathfrak{p}}^{\bdp}$}\label{sec:p-adicL}

In \cite{cas-hsieh1}, the $p$-adic $L$-function introduced by Bertolini--Darmon--Prasanna \cite{bertolini-darmon-prasanna-duke} for $N^-=1$ is shown to be nonzero, and its relation with a $\Lambda$-adic Heegner class via a Perrin-Riou regulator map---an explicit reciprocity law---is established. The aim of this section is to expound these results for a general $N^-$ satisfying (\ref{assu:gen_heeg}). Here we restrict to the weight $2$ case, as this will suffice for our purposes. 

\subsection{Construction of the $p$-adic $L$-function} 

The construction in this section refines work of Brooks \cite{brooks}. We keep the setting and notation introduced in $\S\ref{sec:Sel}$, and assume in addition: 
\[
\textrm{$p\mathcal{O}_K=\mathfrak{p}\overline{\mathfrak{p}}$ splits in $K$.}
\] 
As the case $N^-=1$ is covered in \cite[\S{3}]{cas-hsieh1}, we also assume that $N^-\neq 1$ is a squarefree product of an even number of primes. 

Let $B$ be a quaternion algebra over $\mathbf{Q}$ of discriminant $N^-$, and let $\mathcal{O}_B\subset B$ be a maximal order as in \cite[$\S2.1$]{brooks}. As before, $\widehat{B}=B\otimes_{\mathbf{Z}}\widehat{\mathbf{Z}}$ denotes the profinite completion of $B$, and we put 
\[
\widehat{B}^{(N^-)}=\{b\in\widehat{B}\;:\; x_q=1\;\textrm{for all}\;q\mid N^-\}.
\]
Define $\widehat{\mathbf{Q}}^{(N^-)}$ similarly, and fix an isomorphism $M_2(\widehat{\mathbf{Q}}^{(N^-)})\simeq\widehat{B}^{(N^-)}$. 

Let $\mathrm{Ig}_{N^+,N^-}$  be the Igusa scheme over $\mathbf{Z}_{(p)}$ classifying abelian surfaces with $\mathcal{O}_B$-multiplication and $\Gamma_1(N^+p^\infty)$-level structure. For any valuation ring $W$ finite flat over $\mathbf{Z}_{p}^{\mathrm{ur}}$, denote by $V_p^B(W)$ the space of $p$-adic modular forms over $W$: the space of formal functions on $\mathrm{Ig}_{N^+,N^-}$ over $W$.

Fix a decomposition $N^+\mathcal{O}_K=\mathfrak{N}^+\overline{\mathfrak{N}^+}$ and let $c$ be a positive integer prime to $Np$. Assume for simplicity that $c$ splits in $K$, and fix a decomposition $c\mathcal{O}_K=\mathfrak{C}\overline{\mathfrak{C}}$. Similarly as in (\ref{eq:complex-unif}) and (\ref{eq:CM}), by the complex uniformization
\[
\mathfrak{H}^\pm\times\widehat{B}^\times\rightarrow\mathrm{Ig}_{N^+,N^-}(\mathbf{C}),
\]
the element
\begin{equation}\label{eq:inj}
\xi_c:=\varsigma^{(\infty)}\gamma_c\in\mathrm{GL}_2(\widehat{\mathbf{Q}}^{(N^-)})\simeq\widehat{B}^{(N^-),\times}\hookrightarrow\widehat{B}^\times
\end{equation}
constructed in \cite[p.~577]{cas-hsieh1} defines a CM point $x_c\in\mathrm{Ig}_{N^+,N^-}(\mathbf{C})$ rational over $H_c(\mathfrak{p}^\infty)$, the compositum of $H_c$ with the ray class field of $K$ of conductor $\mathfrak{p}^\infty$. 
By Shimura's reciprocity law, for every $\mathcal{O}_c$-ideal $\mathfrak{a}$ prime to $\mathfrak{N}^+\mathfrak{p}$, letting $a\in\widehat{K}^{(cp),\times}$ be such that $\mathfrak{a}=a\widehat{\mathcal{O}}_c\cap K$ and
\[
\sigma_\mathfrak{a}:=\mathrm{rec}_K(a^{-1})\vert_{H_c(\mathfrak{p}^\infty)}\in\mathrm{Gal}(H_c(\mathfrak{p}^\infty)/K),
\]
the point $x_\mathfrak{a}:=x_c^{\sigma_\mathfrak{a}}$ is defined by the element $\overline{a}^{-1}\xi_c$. 
For $z\in\mathbf{Q}_p$ set
\[
\mathbf{n}(z):=\begin{pmatrix}1 & z\\ 0&1
\end{pmatrix}\in\mathrm{GL}_2(\mathbf{Q}_p)\subset\mathrm{GL}_2(\widehat{\mathbf{Q}}^{(N^-),\times})\hookrightarrow\widehat{B}^\times,
\]
and write $x_\mathfrak{a}*\mathbf{n}(z)$ for the CM point in $\mathrm{Ig}_{N^+,N^-}$ defined by $\overline{a}^{-1}\xi_c\mathbf{n}(z)$.

Let $f_B$ be an automorphic form on $B$ associated to $f$ under the Jacquet--Langlands correspondence, with the $p$-optimal normalization  in \cite[\S{5.1}]{burungale-II}. Consider the ``$p$-depletion'' 
\[
f_B^\flat:=f_B\vert(VU-UV),
\]
where $U$ and $V$ are the Hecke operators defined in e.g. \cite[\S{3.6}]{brooks}. Here we view $f_B$ and $f_B^\flat$ as defined over $\mathscr{O}_\pp$, and let $\widehat{f}_B$ and $\widehat{f}_B^\flat$ be their $p$-adic avatars in $V_p^B(\mathscr{O}_\pp^{\mathrm{ur}})$,  where $\mathscr{O}_\pp^{\mathrm{ur}}$ is the compositum of $\mathscr{O}_\pp$ and $\mathbf{Z}_p^{\mathrm{ur}}$. 

Put $\mathbf{x}_\mathfrak{a}:=x_{\mathfrak{a}}\otimes_{\mathscr{O}_\pp^{\mathrm{ur}}}\overline{\mathbf{F}}_p$ and let $t:\widehat{S}_{\mathbf{x}_\mathfrak{a}}\rightarrow\widehat{\mathbb{G}}_m$ be the Serre--Tate coordinate on the local deformation space $\widehat{S}_{\mathbf{x}_{\mathfrak{a}}}\hookrightarrow\mathrm{Ig}(N)_{/\mathscr{O}_\pp^{\mathrm{ur}}}$. Then $T:=t-1$ gives a canonical uniformizer for the coordinate ring of $\widehat{S}_{\mathbf{x}_{\mathfrak{a}}}$. The Serre--Tate expansion 
\[
\widehat{f}_B(t):=\widehat{f}_B\vert_{\widehat{S}_{\mathbf{x}_{\mathfrak{a}}}}\in\mathscr{O}_\pp^{\mathrm{ur}}\llbracket t-1\rrbracket
\]
defines a $p$-adic measure $\mathrm{d}\widehat{f}_B$ on $\mathbf{Z}_p$ characterized by 
\begin{equation}\label{eq:measure}
\int_{\mathbf{Z}_p}t^x\mathrm{d}\widehat{f}_B(x)=\widehat{f}_B(t).
\end{equation}
The $p$-depletion $\widehat{f}^\flat_B$ defines a measure $\mathrm{d}\widehat{f}_B^\flat$ on $\mathbf{Z}_p$ in the same manner, and it is easily seen that $\mathrm{d}\widehat{f}_B^\flat$ is supported on $\mathbf{Z}_p^\times$. Put 
\[
\widehat{f}_{B,\mathfrak{a}}^\flat(t):=\widehat{f}_{B}^\flat(t^{\mathrm{N}(\mathfrak{a})^{-1}\sqrt{-D_K}^{-1}}),
\]
which again defines a measure on $\mathbf{Z}_p^\times$ characterized as in (\ref{eq:measure}). The following result extends \cite[Prop.~3.3]{cas-hsieh1}.

\begin{prop}\label{prop:3.3}
Let $\phi:\mathbf{Z}_p^\times\rightarrow\mathcal{O}_{\mathbf{Z}_p}^\times$ be a non-trivial finite order character of conductor $p^n$. Then
\[
\int_{\mathbf{Z}_p^\times}\phi(x)t^x\mathrm{d}\widehat{f}_{B,\mathfrak{a}}^\flat(x)=p^{-n}\mathfrak{g}(\phi)\sum_{u\in(\mathbf{Z}/p^n\mathbf{Z})^\times}\phi^{-1}(u)\cdot\widehat{f}_B(x_{\mathfrak{a}}*\mathbf{n}(up^{-n}))
\]
where $\mathfrak{g}(\phi)=\sum_{u\in(\mathbf{Z}/p^n\mathbf{Z})^\times}\phi(u)\zeta_{p^n}^u$ is the Gauss sum.
\end{prop}

\begin{proof}
It suffices to prove an analogue of \cite[Lem.~3.2]{cas-hsieh1} in our setting, for which we shall argue as in \cite[Lem.~4.14]{brooks} to reduce to the elliptic curve case. Indeed, letting $\mathcal{A}_\mathfrak{a}$ denote an abelian surface with $\mathcal{O}_B$-multiplication corresponding to $x_\mathfrak{a}$ under the moduli interpretation of $\mathrm{Ig}_{N^+,N^-}$, by the discussion in \cite[p.~919]{prasanna} there is a degree prime-to-$p$ isogeny
\[
\lambda:\mathcal{A}_{\mathfrak{a}}\rightarrow\mathcal{E}_1\times\mathcal{E}_2,
\]
defined over a number field in which $p$ is unramified, between $\mathcal{A}_\mathfrak{a}$ and the   product of certain CM elliptic curves. Thus from \cite[Prop.~4.1]{brooks} and \cite[Lem.~3.2]{cas-hsieh1} we obtain that if $u\in\mathbf{Z}_p$ then
\[
t(x_{\mathfrak{a}}*\mathbf{n}(up^{-n}))=\zeta_{p^n}^{-u\mathbf{N}(\mathfrak{a})^{-1}\sqrt{-D_K}^{-1}}.
\]
The result now follows from \cite[Lem.~3.1]{cas-hsieh1}.
\end{proof}

Set 
\[
\widetilde{\Lambda}^{\mathrm{ur}}:=\mathscr{O}_\pp^{\mathrm{ur}}\llbracket\mathrm{Gal}(H_{p^\infty}/K)\rrbracket,\quad
\Lambda^{\mathrm{ur}}=\mathscr{O}_\pp^{\mathrm{ur}}\llbracket\mathrm{Gal}(K_\infty/K)\rrbracket. 
\]
Following \cite[Def.~3.7]{cas-hsieh1}, we introduce the following.

\begin{defn}\label{def:bdp}
Let $\psi$ be an auxiliary anticyclotomic Hecke character of $K$ of infinity type $(1,-1)$ and conductor $c$.
\begin{enumerate}
\item{} 
Let $\mathscr{L}_{\mathfrak{p},\psi}\in\widetilde{\Lambda}^{\mathrm{ur}}$ be the $p$-adic measure on $\mathrm{Gal}(H_{p^\infty}/K)$ defined by
\[
\mathscr{L}_{\mathfrak{p},\psi}(\xi)=
\sum_{[\mathfrak{a}]\in\mathrm{Pic}(\mathcal{O}_c)}
\psi(\mathfrak{a})\mathbf{N}(\mathfrak{a})^{-1}
\int_{\mathbf{Z}_p^\times}\psi_\mathfrak{p}(x)\xi(\mathrm{rec}_\mathfrak{p}(x)\sigma_\mathfrak{a}^{-1})\;\mathrm{d}\widehat{f}^\flat_{B,\mathfrak{a}}(x_{\mathfrak{a}})
\]
for all $\xi:\mathrm{Gal}(H_{p^\infty}/K)\rightarrow\mathcal{O}_{\mathbf{C}_p}^\times$, where $\psi_\mathfrak{p}$
is the component of $\psi$ at $\mathfrak{p}$ and  $\mathrm{rec}_{\mathfrak{p}}:K_\mathfrak{p}^\times\rightarrow\mathrm{Gal}(H_{p^\infty}/K)$ is the local reciprocity map. 
\item{} Let $\mathrm{tw}_\phi:\widetilde{\Lambda}^{\mathrm{ur}}\rightarrow\widetilde{\Lambda}^{\mathrm{ur}}$ be the map defined by $\gamma\mapsto\phi(\gamma)\gamma$ for $\gamma\in\mathrm{Gal}(H_{p^\infty}/K)$.  The $p$-adic $L$-function 
\[
\mathscr{L}_{\mathfrak{p}}^{\bdp}\in\Lambda^{\mathrm{ur}}
\] 
is the image of $\mathrm{tw}_{\psi^{-1}}(\mathscr{L}_{\mathfrak{p},\psi})$ under the natural projection $\widetilde{\Lambda}^{\mathrm{ur}}\rightarrow\Lambda^{\mathrm{ur}}$. 
\end{enumerate}
\end{defn}

\begin{rem}\label{rem:walds}
In light of the Waldspurger formula, the \emph{square} of $\mathscr{L}_\mathfrak{p}^\bdp$ is expected to interpolate the central critical $L$-values $L(f/K,\xi,1)$ for the Rankin--Selberg convolution of $f$ with theta series of weight $\ell\geqslant 3$ attached to certain anticyclotomic Hecke characters $\xi$. For $N^-=1$, this interpolation property is shown in \cite[Prop.~3.8]{cas-hsieh1} as a consequence of results in \cite[Thm.~A]{hsieh-doc}; for $N^-\neq 1$, based on Prasanna's explicit Waldspurger formula \cite[Thm.~3.2]{prasanna}, 
the interpolation is deduced in \cite[\S{8}]{brooks} for squarefree $N$ and $\xi$ crystalline at the primes above $p$. 
\end{rem}

\subsection{Explicit reciprocity law}

The next result relates the $p$-adic $L$-function $\mathscr{L}_\mathfrak{p}^{\bdp}$ to the element $\kappa_\infty$ in $(\ref{eq:dist})$. Recall that $\overline{\rho}:G_\mathbf{Q}\rightarrow\mathrm{GL}_2(\mathbf{F})$ denotes the Galois representation afforded by $A_f[\wp]$. In the following, we use the  superscript ``ur'' to denote extension of scalars to $\mathbf{Z}_p^{\mathrm{ur}}$.

\begin{thm}
\label{thm:cas-hsieh}
Suppose $\overline{\rho}\vert_{G_K}$ is absolutely irreducible. There exists an injective $\Lambda^{\mathrm{ur}}$-linear map
\[
\mathrm{Log}_{\mathfrak{p}}:\mathrm{H}^1_{}(K_\mathfrak{p},\mathrm{Fil}_\mathfrak{p}^+(\mathbf{T}))^{\mathrm{ur}}\hookrightarrow\Lambda^{\mathrm{ur}}
\]
with finite cokernel such that
\begin{equation}\label{eq:ERL}
\mathrm{Log}_{\mathfrak{p}}(\mathrm{loc}_{\mathfrak{p}}(\kappa_\infty))=-\mathscr{L}_{\mathfrak{p}}^{\bdp}\cdot\sigma_{-1,\mathfrak{p}},
\end{equation}
where $\sigma_{-1,\mathfrak{p}}\in\mathrm{Gal}(K_\infty/K)$ has order two. 
\end{thm}

\begin{proof}
The construction of the map $\mathrm{Log}_{\mathfrak{p}}$ is given in \cite[Thm.~5.1]{cas-hsieh1} (note that
the injectivity of this map is not explicitly stated in \emph{loc.cit.}, but it follows from \cite[Prop.~4.11]{LZ2}). 
For $N^-=1$, $(\ref{eq:ERL})$ is just the weight $2$ case of the ``explicit reciprocity law'' in \cite[Thm.~5.7]{cas-hsieh1}. We explain how to extend that result to the case at hand. 

Let $\chi:\mathrm{Gal}(K_\infty/K)\rightarrow\mu_{p^\infty}$ be the $p$-adic avatar of ring class character of conductor $p^n\mathcal{O}_K$, with $n>1$. Following the calculations in \cite[pp.~598-9]{cas-hsieh1} (with Proposition~3.3 in \emph{loc.cit.} replaced by the above Proposition~\ref{prop:3.3}) we find: 
\begin{equation}\label{eq:calc}
\begin{split}
\mathscr{L}_{\mathfrak{p}}^{\bdp}(\chi^{-1})&=
\sum_{[\mathfrak{a}]\in\mathrm{Pic}(\mathcal{O}_c)}\chi^{-1}(\mathfrak{a})\cdot\bigl(\theta^{-1}\widehat{f}^\flat_B\otimes\chi_{\mathfrak{p}}^{-1}\bigr)(x_\mathfrak{a})\\
&=p^{-n}\mathfrak{g}(\chi_\mathfrak{p}^{-1})\chi_{\mathfrak{p}}(p^n)
\sum_{\sigma\in\mathrm{Gal}(H_{cp^n}/K)}\chi(\sigma)\cdot\theta^{-1}\widehat{f}^\flat_B(x_{cp^n}^\sigma).
\end{split}
\end{equation}
Here 
\[
\theta^{-1}\widehat{f}^\flat_B:=\lim_{i\to\infty}\theta^{-1+p^i(p-1)}\widehat{f}_B
\]
where $\theta$ is the Katz $p$-adic differential operator acting  as $t\frac{d}{dt}$ on the $t$-expansions. Since by \cite[Prop.~7.4]{brooks} and \cite[Prop.~A.0.1]{LZZ}, the term $\theta^{-1}\widehat{f}^\flat_B(x_{cp^n}^\sigma)$ computes the image of 
\[
\iota_1(x_1(cp^n))\in J(X_{N^+,N^-})(H_{cp^n})\otimes_{\mathbf{Z}}\mathscr{O}_\pp
\]
(\emph{cf.} (\ref{eq:CM}), (\ref{eq:emb})) under the formal group logarithm
\[
\mathrm{log}_{\omega_f}:J(X_{N^+,N^-})(H_{cp^n,w})\otimes_{\mathbf{Z}}\mathscr{O}_\pp\rightarrow H_{cp^n,w}
\] 
associated to the differential $\omega_f$, where $H_{cp^n,w}$ denotes the completion of $H_{cp^n}$ at the prime above $\mathfrak{p}$ induced by our fixed embedding $\iota_p$, substituting this into $(\ref{eq:calc})$ the argument in \cite[Thm.~5.7]{cas-hsieh1} applies verbatim to yield the  proof of (\ref{eq:ERL}). 
\end{proof}

\begin{cor}\label{cor:Lneq0}
Suppose $\overline{\rho}\vert_{G_K}$ is absolutely irreducible. Then $\mathrm{loc}_{\mathfrak{p}}(\kappa_\infty)$ is not $\Lambda$-torsion. In particular, the $p$-adic $L$-function $\mathscr{L}_\mathfrak{p}^{\bdp}$ is nonzero.
\end{cor}

\begin{proof}
Recall that the class $\kappa_\infty$ is nonzero by \cite{cornut-vatsal}. Therefore for all but finitely many characters $\chi:\mathrm{Gal}(K_\infty/K)\rightarrow\mathcal{O}_{\mathbf{C}_p}^\times$ factoring through $\mathrm{Gal}(K_n/K)$ for some $n\geqslant 0$, the image $\kappa^\chi$ of $\kappa_\infty$ under the specialization map
\[
\mathrm{H}^1(K,\mathbf{T})\rightarrow\mathrm{H}^1(K,T\otimes\chi)\simeq\mathrm{H}^1(K_n,T)^{(\chi)}
\]
is nonzero. Here $\mathrm{H}^1(K_n,T)^{(\chi)}$ denote the $\chi$-isotypic component of $\mathrm{H}^1(K_n,T)$ under the action of $\mathrm{Gal}(K_n/K)$. Since by construction $\kappa^\chi$ arises as the image of the twisted Heegner point
\[
y_\chi:=\alpha_p^{-n}\sum_{\sigma\in\mathrm{Gal}(H_{p^n}/K)}\chi^{-1}(\sigma)\otimes\iota_1(x_1(p^n))^\sigma\in E(K_n)^{(\chi)},
\]
by \cite[Thm.~3.2]{nekovar-euler-systems} it follows that if $\kappa^\chi\neq 0$ then both $\sha(E/K_n)^{(\chi)}$ and the quotient of $E(K_n)^{(\chi)}$ by the submodule generated by $y_\chi$ are finite. Since $E(K_n)$ injects into $E(K_{n,v})$, it follows that for any prime $v$ of $K_n$ above $\mathfrak{p}$, we have the implication
\[
\kappa^\chi\neq 0\;\Longrightarrow\;\mathrm{loc}_v(\kappa^\chi)\neq 0.
\]
Letting $\chi$ as above vary, this shows that $\mathrm{loc}_\mathfrak{p}(\kappa_\infty)$ is not $\Lambda$-torsion. The last claim in the corollary then follows from Theorem~\ref{thm:cas-hsieh}.
\end{proof}

\begin{rem}
The nonvanishing of $\mathscr{L}_{\mathfrak{p}}^{\bdp}$ can also be shown following Hida's methods. For $N^-=1$, this is done in \cite[Thm.~3.9]{cas-hsieh1} as an application of \cite[Thm.~C]{hsieh-doc}, and for $N^-\neq 1$ the result can be similarly deduced from \cite{burungale-II}.
\end{rem}

\section{Proof of Theorem~B}\label{sec:equiv-IMC}

As in the preceding section, we keep the setting and notations introduced in $\S\ref{sec:Sel}$, and assume in addition that $p\mathcal{O}_K=\mathfrak{p}\overline{\mathfrak{p}}$ splits in $K$.

Consider the following variants of the Selmer groups $\mathrm{Sel}(K,\mathbf{T})$ and $\mathrm{Sel}(K,\mathbf{W})$ in $\S\ref{sec:Sel}$ obtained by changing the local condition at the primes above $p$. Let $M$ denote either $\mathbf{T}$ or $\mathbf{W}$. For $w$ a prime of $K$ above $p$, set
\begin{align*}
\mathrm{H}^1_{\emptyset}(K_w, M) & = \mathrm{H}^1(K_w, M), \\
\mathrm{H}^1_{\mathrm{ord}}(K_w, M) & = \mathrm{H}^1(K_w,\mathrm{Fil}^+(M)), \\
\mathrm{H}^1_{0}(K_w, M) & = \{0\},
\end{align*}
and for $\bullet,\circ \in \lbrace \emptyset, \mathrm{ord}, 0 \rbrace$ define
\[
\mathrm{Sel}_{\bullet,\circ}(K,M) 
:= \mathrm{ker} \biggr\{ 
\mathrm{H}^1(K_{\Sigma}/K,M) \to \dfrac{\mathrm{H}^1(K_{\mathfrak{p}}, M)}{\mathrm{H}^1_{\bullet}(K_{\mathfrak{p}},M)} \times \dfrac{\mathrm{H}^1(K_{\overline{\mathfrak{p}}},M)}{\mathrm{H}^1_{\circ}(K_{\overline{\mathfrak{p}}},M)} \times \prod_{w \in \Sigma, w \nmid p} \dfrac{\mathrm{H}^1(K_w,M)}{\mathrm{H}^1_{\mathcal{F}_{\mathrm{ord}}}(K_w,M)} \biggl\}.
\]

In particular, $\mathrm{Sel}_{\mathrm{ord},\mathrm{ord}}(K,M)$ is the same as the earlier $\mathrm{Sel}(K,M)$. For the ease of notation, we also set 
\[
\mathcal{S}_{\bullet,\circ}=\mathrm{Sel}_{\bullet,\circ}(K,\mathbf{T}),\quad
\mathcal{X}_{\bullet,\circ}:=\mathrm{Sel}_{\bullet,\circ}(K,\mathbf{W})^\vee,
\]
so $\mathcal{X}_{\mathrm{ord},\mathrm{ord}}$ is the same as the earlier $\mathcal{X}$.

Denote by $\mathscr{O}_\pp^{\mathrm{ur}}$ the compositum of $\mathscr{O}$ with $\mathbf{Z}_p^{\mathrm{ur}}$, and set 
\[
\Lambda^{\mathrm{ur}}=\mathscr{O}_\pp^{\mathrm{ur}}\llbracket\mathrm{Gal}(K_\infty/K)\rrbracket.
\]
In this section we prove the following result, which implies Theorem~B in the introduction.

\begin{thm}\label{thm:main-BDP}
Suppose $p\nmid 6N$ and $\pp$ is a prime of $\mathscr{O}$ above $p$ such that the following hold:
\begin{itemize}
\item $f$ is ordinary at $\pp$,
\item Hypothesis~$\heartsuit$ holds for $(f,\pp,K)$, 
\item $\overline{\rho}_0$ is surjective,
\item $p$ is non-anomalous.
\end{itemize}
Then $\mathcal{X}_{\emptyset,0}$ is $\Lambda$-torsion, and
\[
\mathrm{Char}_\Lambda(\mathcal{X}_{\emptyset,0})=(\mathscr{L}_{\mathfrak{p}}^{\bdp})^2
\]
as ideals in $\Lambda^{\mathrm{ur}}$. 
\end{thm}

After Theorem~\ref{thm:main}, the proof will be an immediate consequence of the next result, showing in particular that Conjecture~\ref{conj:BDP} is equivalent to Conjecture~\ref{conj:HPMC} when $p$ splits in $K$.

\begin{thm}\label{thm:equiv-imc}
Assume that $\mathrm{H}^0(G_K,\overline{\rho}_0)=0$. Then the following are equivalent:
\begin{enumerate}
\item[(i)] Both $\mathcal{S}$ and 
$\mathcal{X}$ have $\Lambda$-rank one, and the following divisibility holds in $\Lambda$:
\[
\mathrm{Char}_\Lambda(\mathcal{X}_{\mathrm{tors}})\supset\mathrm{Char}_\Lambda\bigl(\mathcal{S}/\Lambda\kappa_\infty\bigr)^2.
\]
\item[(ii)] Both $\mathcal{S}_{0,\emptyset}$ and 
$\mathcal{X}_{\emptyset,0}$ are $\Lambda$-torsion, and the following divisibility holds in $\Lambda^{\mathrm{ur}}$: 
\[
\mathrm{Char}_\Lambda(\mathcal{X}_{\emptyset,0})\supset(\mathscr{L}_{\mathfrak{p}}^{\bdp})^2.
\]
\end{enumerate}
Moreover, the same result holds for the opposite divisibilities.
\end{thm}

\begin{proof}
We begin by noting that, since $\mathrm{Gal}(K_\infty/K)$ is a pro-$p$ group, our hypothesis implies that $\mathrm{H}^0(G_{K_\infty},\bar{\rho}_0)=0$, and so $\mathrm{H}^1(K_\Sigma/K,\mathbf{T})$ is torsion-free by \cite[\S{1.3.3}]{perrin-riou-book}. As we shall explain in the next paragraphs, the Theorem can be extracted from \cite[App.~A]{castella-beilinson-flach}; all the references in the rest of this proof will be to results in that appendix. (Note that the Selmer groups considered in \cite{castella-beilinson-flach} have the unramified local condition at all primes $w\nmid p$, but the same arguments apply \emph{verbatim} to the Selmer groups we consider here. Note also, although this is not needed for our arguments, that by \cite[\S{5}]{pw-mu} both Selmer groups are the same if $\overline{\rho}$ is ramified at all primes $\ell\mid N^-$.) 

We first show that $\mathcal{X}$ has $\Lambda$-rank one if and only if $\mathcal{X}_{\emptyset,0}$ is $\Lambda$-torsion. 	
If $\mathcal{X}$ has $\Lambda$-rank one, then $\mathcal{S}$ has $\Lambda$-rank one by Lemma~2.3(1), and so  $\mathcal{X}_{\emptyset,0}$ is $\Lambda$-torsion by Lemma~A.4. Conversely, if $\mathcal{X}_{\emptyset,0}$ is $\Lambda$-torsion, then $\mathcal{X}_{\mathrm{ord},0}$ is also $\Lambda$-torsion (see eq.~(A.7)), and so $X_{\mathrm{ord},\emptyset}$ has $\Lambda$-rank one by Lemma~2.3(2). Global duality yields the exact sequence
\begin{equation}\label{eq:PT}
0\rightarrow\mathrm{coker}\bigl(\mathrm{loc}_{\mathfrak{p}}:\mathcal{S}\rightarrow \mathrm{H}^1_{\mathrm{ord}}(K_{\mathfrak{p}},\mathbf{T})\bigr)\rightarrow
\mathcal{X}_{\emptyset,\mathrm{ord}}\rightarrow
\mathcal{X}\rightarrow 0.
\end{equation}
Since $\mathrm{H}^1_{\mathrm{ord}}(K_{\mathfrak{p}},\mathbf{T})$ has $\Lambda$-rank one, the left term in this sequence is $\Lambda$-torsion by Theorem~A.1 and the nonvanishing of $\mathscr{L}_\mathfrak{p}^{\bdp}$, and since the right term is isomorphic to $\mathcal{X}_{\mathrm{ord},\emptyset}$ by the action complex conjugation, and hence is of $\Lambda$-rank one by the above analysis, we conclude from $(\ref{eq:PT})$ that $\mathcal{X}$ also has $\Lambda$-rank one. 
	
To relate the divisibilities, assume that $\mathcal{X}$ has $\Lambda$-rank one. By Lemma~2.3(1), this amounts to the assumption that $\mathcal{S}$ has $\Lambda$-rank one, and so by Lemma~A.3 for every height one prime $\mathfrak{P}$ of $\Lambda$ we have
	\begin{equation}\label{eq:A3}
	\mathrm{length}_{\mathfrak{P}}(\mathcal{X}_{\emptyset,0})=\mathrm{length}_{\mathfrak{P}}(\mathcal{X}_{\mathrm{tors}})+2\;\mathrm{length}_{\mathfrak{P}}(\mathrm{coker}(\mathrm{loc}_{\mathfrak{p}})),
	\end{equation} 
and by Lemma~A.4 for every height one prime $\mathfrak{P}'$ of $\Lambda^{\mathrm{ur}}$  we have
	\begin{equation}\label{eq:A4}
	\mathrm{ord}_{\mathfrak{P}'}(\mathscr{L}_{\mathfrak{p}}^{\bdp})=\mathrm{length}_{\mathfrak{P}'}(\mathrm{coker}(\mathrm{loc}_{\mathfrak{p}})\Lambda^{\mathrm{ur}})+\mathrm{length}_{\mathfrak{P}'}\bigl(\mathcal{S}^{\mathrm{ur}}/\Lambda^{\mathrm{ur}}\kappa_\infty\bigr),
	\end{equation}
	where $\mathcal{S}^{\mathrm{ur}}$ denotes the extension of scalars of $\mathcal{S}$ to $\Lambda^{\mathrm{ur}}$. Thus for any height one prime $\mathfrak{P}\subset\Lambda$, letting $\mathfrak{P}'$ denote its extension to $\Lambda^{\mathrm{ur}}$, we see from $(\ref{eq:A3})$ and $(\ref{eq:A4})$ that
	\[
	\mathrm{length}_{\mathfrak{P}}(\mathcal{X}_{\mathrm{tors}})\leqslant 2\;\mathrm{length}_{\mathfrak{P}}\bigl(\mathcal{S}/\Lambda\kappa_\infty\bigr)\quad\Longleftrightarrow\quad
	\mathrm{length}_{\mathfrak{P}}(\mathcal{X}_{\emptyset,0})\leqslant 2\;\mathrm{ord}_{\mathfrak{P}'}(\mathscr{L}_{\mathfrak{p}}^{\bdp}),
	\]
	and similarly for the opposite inequalities. The result follows. 
\end{proof}

\begin{proof}[Proof of Theorem~\ref{thm:main-BDP}]
Since Theorem~\ref{thm:main} holds under the given hypotheses, the result follows from the equivalence in Theorem~\ref{thm:equiv-imc}. 
\end{proof}

\appendix

\section{An alternative approach in rank one}

In this appendix we give an alternative proof\footnote{Appeared in an earlier version of this paper released in June 2018; see \cite{bck-old}.} of the following special case of Theorem~\ref{thm:main} and Theorem~\ref{thm:main-BDP}, but which does \emph{not} require the hypothesis that $p$ is non-anomalous.

\begin{thm}\label{thm:main-rank1}
Let $E/\mathbf{Q}$ be an elliptic curve of condition $N$, let $p>3$ be a prime where $E$ has good ordinary reduction, and let $K$ an imaginary quadratic field of discriminant prime to $Np$ satisfying hypotheses $\mathrm{(\ref{assu:disc})}$ and $\mathrm{(\ref{assu:gen_heeg})}$. Suppose in addition that the following conditions hold: 
\begin{itemize}
\item $(E, K ,p)$ satisfies Hypothesis~$\spadesuit$,
\item $\overline{\rho}:G_\mathbf{Q}\rightarrow\mathrm{Aut}_{\mathbf{F}_p}(E[p])$ is surjective,  
\item $p\mathcal{O}_K = \mathfrak{p}\overline{\mathfrak{p}}$ splits in $K$, 
\item $\mathrm{ord}_{s=1}L(E/K,s)=1$. 
\end{itemize}
Then Conjecture~\ref{conj:HPMC} and Conjecture~\ref{conj:BDP} hold.
\end{thm}

The proof of Theorem~\ref{thm:main-rank1} will occupy the remainder of this appendix. After possibly changing $E$ within its isogeny class, we shall assume that $E$ is ``$(\mathbf{Z},p\mathbf{Z}_p)$-optimal'' in the sense of \cite[\S{3.7}]{wei-zhang-mazur-tate}. Denote by $H_K$ the Hilbert class field of $K$, and let
\[
x_1\in\mathrm{CM}(X_{N^+,N^-})\cap X_{N^+,N^-}(H_K)
\]
be the Heegner point constructed in $(\ref{eq:CM})$ (i.e., taking $m=p^n=1$). Letting $\pi:J(X_{N^+,N^-})\rightarrow E$ be the quotient map, we set
\[
z_1:=\pi(\iota_1(x_1))\in E(H_K)\otimes_{\mathbf{Z}}\mathbf{Z}_p,
\]
where $\iota_1$ is as in $(\ref{eq:emb})$. By the Gross--Zagier formula \cite{gross-zagier-original, yuan-zhang-zhang, explicit-gross-zagier-waldspurger} we have
\begin{equation}\label{eq:GZ}
L'(E/K,1)\neq 0\quad\Longleftrightarrow\quad z_K:=\textrm{Tr}_{H_K/K}(z_1)\;\textrm{in non-torsion}.
\end{equation}

Upon the choice of a topological generator $\gamma\in\mathrm{Gal}(K_\infty/K)$, it will be convenient to view the $p$-adic $L$-function $\mathscr{L}_{\mathfrak{p}}^{\bdp}$ of $\S\ref{sec:p-adicL}$ as an element in the power series ring $\mathbf{Z}_p^{\mathrm{ur}}\llbracket{T}\rrbracket$ via the isomorphism $\Lambda^{\mathrm{ur}}\simeq\mathbf{Z}_p^{\mathrm{ur}}\llbracket{T}\rrbracket$ sending $\gamma-1\mapsto T$.

The starting point of the approach in this appendix is the $p$-adic Waldspurger formula due to Bertolini--Darmon--Prasanna \cite{bertolini-darmon-prasanna-duke}, which corresponds to the specialization of $(\ref{eq:ERL})$ at the trivial character of $\mathrm{Gal}(K_\infty/K)$.

\begin{prop}[Bertolini--Darmon--Prasanna, Brooks]\label{thm:bdp}
Assume that:
\begin{itemize}
\item $p\mathcal{O}_K = \mathfrak{p}\overline{\mathfrak{p}}$ splits in $K$, 
\item $E[p]$ is irreducible as a $G_K$-module. 
\end{itemize}
Then 
\[
\mathscr{L}^{\bdp}_{\mathfrak{p}}(0) = \biggl( \frac{1-a_p+p}{p} \biggr) \cdot \bigl( \mathrm{log}_{\omega_E} z_K \bigr),
\]
where  the equality is up to a $p$-adic unit.
\end{prop}

\begin{proof}
This is a special case of \cite[Thm.~5.13]{bertolini-darmon-prasanna-duke} ($N^-=1$) and \cite[Thm.~1.1]{brooks} ($N^-\neq 1$), as explained in Propositions~5.1.6 and 5.1.7 of \cite{jetchev-skinner-wan}. 
\end{proof}

The following result is a consequence of the ``anticyclotomic control theorem'' of \cite[\S{3}]{jetchev-skinner-wan}.

\begin{prop}\label{thm:equiv-spval}
Assume that: 
\begin{itemize}
\item{} $p\mathcal{O}_K=\mathfrak{p}\overline{\mathfrak{p}}$ splits in $K$, 
\item{} $E[p]$ is irreducible as a $G_K$-module, 
\item{} $\mathrm{rank}_{\mathbf{Z}}E(K)=1$ and $\#\sha(E/K)[p^\infty]<\infty$,
\end{itemize}
and let $P\in E(K)$ be a point of infinite order.  
Then $\mathcal{X}_{\emptyset, 0}$ is $\Lambda$-torsion, and letting
\[
f_{\emptyset,0}(T)\in\mathbf{Z}_p\llbracket T \rrbracket
\]
be a generator of the characteristic ideal of $\mathcal{X}_{\emptyset,0}$, the following equivalence holds:
	\[
	f_{\emptyset,0}(0)\sim_p\biggl(\dfrac{1-a_p+p}{p}\biggr)^2\cdot\log_{\omega_E}(P)^2\quad\Longleftrightarrow\quad[E(K):\mathbf{Z}.P]^2\sim_p\#\sha(E/K)[p^\infty]\prod_{\ell\mid N^+}c_\ell^2,
	\]
where  $c_\ell$ is the Tamagawa number of $E/\mathbf{Q}_\ell$, and $\sim_p$ denotes equality up to a $p$-adic unit.
\end{prop}

\begin{proof}
As shown in \cite[pp.~395-6]{jetchev-skinner-wan}, our assumptions imply hypotheses (corank~1), (sur), and (irred$_{\mathcal{K}}$) of \cite[\S{3.1}]{jetchev-skinner-wan}, and so by [\emph{op.cit.}, Thm.~3.3.1] (with $S=S_p$ the set of primes dividing $N$ and $\Sigma=\emptyset$) the module $\mathcal{X}_{\emptyset, 0}$ is $\Lambda$-torsion, and  
\begin{equation}\label{eq:control}
\# \mathbf{Z}_p / f_{\emptyset, 0}(0) = \#\mathrm{H}^1_{\mathcal{F}_{\mathrm{ac}}}(K,E[p^\infty]) \cdot C^{}(E[p^\infty]),
\end{equation}
where $\mathrm{H}^1_{\mathcal{F}_{\mathrm{ac}}}(K,E[p^\infty])$ is the anticyclotomic Selmer group introduced in \cite[\S{2.2.3}]{jetchev-skinner-wan} and 
\[
C^{}(E[p^\infty]) := \# \mathrm{H}^0(K_{\mathfrak{p}}, E[p^\infty]) \cdot \# \mathrm{H}^0(K_{\overline{\mathfrak{p}}}, E[p^\infty]) \cdot \prod_{w\mid N^+} \# \mathrm{H}^1_{\mathrm{ur}}(K_w, E[p^\infty]).
\]

Under our hypotheses, by [\emph{op.cit.}, (3.5.d)] the Selmer group $\mathrm{H}^1_{\mathcal{F}_{\mathrm{ac}}}(K,E[p^\infty])$ is finite, with order given by 
\begin{equation}\label{eq:ev-index}
	\#\mathrm{H}^1_{\mathcal{F}_{\mathrm{ac}}}(K,E[p^\infty])
	=
	\# \sha(E/K)[p^\infty] \cdot \biggl(  \dfrac{ \# (\mathbf{Z}_p / ( \frac{1-a_p+p}{p}) \cdot \mathrm{log}_{\omega_E} P) }{ [E(K): \mathbf{Z}.P]_p \cdot \# \mathrm{H}^0(K_{\mathfrak{p}}, E[p^\infty]) } \biggr)^2,
	\end{equation}
where $[E(K): \mathbf{Z}.P]_p$ 
	denotes the $p$-part of the index $[E(K): \mathbf{Z}.P]$. Combining $(\ref{eq:control})$ and $(\ref{eq:ev-index})$ we thus arrive at
	\[
	\# \mathbf{Z}_p / f_{\emptyset, 0}(0) = 
	\# \sha(E/K)[p^\infty]\cdot\biggl(  \dfrac{ \# (\mathbf{Z}_p / ( \frac{1-a_p+p}{p}) \cdot \mathrm{log}_{\omega_E} P) }{ [E(K): \mathbf{Z}.P]_p }\biggr)^2\cdot\prod_{w\mid N^+} \# \mathrm{H}^1_{\mathrm{ur}}(K_w, E[p^\infty]).
	\]
	
Since $\#\mathrm{H}^1_{\mathrm{ur}}(K_w,E[p^\infty])$ is the $p$-part of the Tamagawa number of $E/K_w$ (see e.g. \cite[Lem.~9.1]{skinner-zhang}) and the primes $\ell\mid N^+$ split in $K$, 
the result follows.
\end{proof}

By construction, the point $z_K$ in $(\ref{eq:GZ})$ is a $p$-adic unit multiple of a point $y_K\in E(K)$, so by (\ref{eq:GZ}) $y_K$ has infinite order if and only if $L'(E/K,1)\neq 0$.

\begin{thm}[Kolyvagin, W.~Zhang]\label{thm:ppart}
Assume that:
\begin{itemize}
\item $p\nmid 6N$ is a prime where $E$ has ordinary reduction,
\item $(E,p,K)$ satisfies Hypothesis~$\spadesuit$,
\item $\overline{\rho}$ is surjective,
\item $y_K\in E(K)$ has infinite order. 
\end{itemize}
Then for all primes $\ell\mid N^+$ the Tamagawa numbers $c_\ell$ are $p$-adic units and
\begin{equation}\label{eq:ppartBSD}
\#\sha(E/K)[p^\infty]\sim_p[E(K):\mathbf{Z}.y_K]^2.
\end{equation}
\end{thm}

\begin{proof}
For the first claim, note that if $\ell\Vert N^+$ then $c_\ell$ is a $p$-adic unit by part (i) of Hypothesis~$\spadesuit$ in the introduction, while if $\ell^2\mid N^+$ then by \cite[Lem.~5.1(2)]{wei-zhang-mazur-tate} the group $\mathrm{H}^1(\mathbf{Q}_\ell,E[p])$ vanishes, so from \cite[Lem.~6.3]{wei-zhang-mazur-tate} we see that $c_\ell$ is also a $p$-adic unit. 

On the other hand, as shown in \cite[Thm.~10.2]{wei-zhang-mazur-tate}, equality~(\ref{eq:ppartBSD}) follows from Kolyvagin's structure theorem for $\sha(E/K)[p^\infty]$ and the proof of Kolyvagin's conjecture, \cite[Thm.~9.3]{wei-zhang-mazur-tate}.
\end{proof}

The last ingredient we need is the following useful commutative algebra result from  \cite{skinner-urban}.

\begin{lem}\label{lem:3.2}
Let $A$ be a local ring, and assume that $\mathfrak{a}\subset A$ is a proper ideal such  that $A/\mathfrak{a}$ is an integral domain. Let $L\in A$, let $I\subset A$ be an ideal contained in $(L)$, and denote by $\overline{L}$ and $\overline{I}$ their reductions modulo $\mathfrak{a}$. If $\overline{L}\neq 0$ and $\overline{L} \in \overline{I}$, then $I = (L)$.
\end{lem}

\begin{proof}
This is a special case of \cite[Lem.~3.2]{skinner-urban}.
\end{proof}

For our application, we shall take $A = \Lambda$, $\mathfrak{a}\subset \Lambda$ the augmentation ideal, $L$ a characteristic power series for $\mathcal{X}_{\emptyset,0}$, and $I$ the ideal generated by the square of $\mathscr{L}_{\mathfrak{p}}^{\bdp}$. 

\begin{proof}[Proof of Theorem~\ref{thm:main-rank1}]
By the first part of Theorem~\ref{thm:bipartite}, both $\mathcal{S}$ and $\mathcal{X}$ have $\Lambda$-rank one, and the $\Lambda$-torsion submodule $\mathcal{X}_{\mathrm{tors}}$ of $\mathcal{X}$ is such that
\begin{equation}\label{eq:howard-div}
\mathrm{Char}_\Lambda(\mathcal{X}_{\mathrm{tors}})\supset\mathrm{Char}_\Lambda\bigr(\mathcal{S}/\Lambda\kappa_\infty\bigr)^2.
\end{equation}
By Theorem~\ref{thm:equiv-imc}, this implies that $\mathcal{X}_{\emptyset,0}$ is $\Lambda$-torsion, with
\begin{equation}\label{eq:howard-div-bdp}
\mathrm{Char}_\Lambda(\mathcal{X}^{}_{\emptyset,0})\supset(\mathscr{L}_{\mathfrak{p}}^{\bdp})^2
\end{equation}
as ideals in $\Lambda^{\mathrm{ur}}$. Let $f_{\emptyset,0}(T)\in\Lambda$ be a characteristic power series for $\mathcal{X}^{}_{\emptyset,0}$, viewed as an element in $\mathbf{Z}_p\llbracket T\rrbracket$. Since $L'(E/K,1)\neq 0$ by hypothesis, $y_K\in E(K)$ has infinite order by the Gross--Zagier formula, and from Theorem~\ref{thm:ppart} and Proposition~\ref{thm:equiv-spval} we deduce that 
\[
f_{\emptyset,0}(0)\sim_p\biggl(\dfrac{1-a_p+p}{p}\biggr)^2\cdot\log_{\omega_E}(y_K)^2,
\]
which in particular shows that $f_{\emptyset,0}(0)\neq 0$. Since the points $y_K$ and $z_K$ differ by a $p$-adic unit, by Proposition~\ref{thm:bdp} it follows that
\begin{equation}\label{eq:equiv-spval}
f_{\emptyset,0}(0)\sim_p\mathscr{L}_{\mathfrak{p}}^{\bdp}(0)^2,\quad\quad\textrm{with $f_{\emptyset,0}(0)\neq 0$.}
\end{equation} 
 
In light of Lemma~\ref{lem:3.2}, from $(\ref{eq:equiv-spval})$ we deduce that the divisibility $(\ref{eq:howard-div-bdp})$ is an equality, yielding the proof of Conjecture~\ref{conj:BDP}. By Theorem~\ref{thm:equiv-imc}, the equality in (\ref{eq:howard-div}), and hence Conjecture~\ref{conj:HPMC} follows from this, concluding the proof of Theorem~\ref{thm:main-rank1}.
\end{proof}

\begin{rem}\label{rem:KS}
In the terminology of \cite{mazur-rubin-book}, W.~Zhang's theorem \cite[Thm.~9.3]{wei-zhang-mazur-tate} on Kolyvagin's conjecture may be interpreted as establishing \emph{primitivity}  
of the ``Heegner point Kolyvagin systems'' $\{\kappa_n\}_n$ constructed by Howard \cite[$\S{1.7}$]{howard-kolyvagin}, \cite[\S{2.3}]{howard-gl2-type}.  
Mazur--Rubin also introduced the 
notion of $\Lambda$-\emph{primitivity}, 
and letting $\{\boldsymbol{\kappa}_n\}_n$ be Howard's $\Lambda$-adic Heegner point Kolyvagin system \cite[$\S{2.3}$]{howard-kolyvagin}, \cite[$\S{3.4}$]{howard-gl2-type}, our approach to Theorem~\ref{thm:main-rank1} may be seen as a realization\footnote{In his Princeton senior thesis, M.~Zanarella \cite{murilo} has developed this idea to obtain a different proof of our Theorem~\ref{thm:main} by incorporating primitivity into Howard's theory of Kolyvagin systems \cite{howard-kolyvagin, howard-gl2-type}.} of the implications
\[
\textrm{$\{\kappa_n\}_n$ is primitive}\;\Longrightarrow\;
\textrm{$\{\boldsymbol{\kappa}\}_n$ is $\Lambda$-primitive}\;\Longrightarrow\;
\textrm{Conjecture~\ref{conj:HPMC} holds},
\]
which then also implies Conjecture~\ref{conj:BDP} by the equivalence in Theorem~\ref{thm:equiv-imc}.
\end{rem}

\begin{rem}\label{rem:higher-rank} 
The assumption that $L'(E/K,1)\neq 0$ should not be essential to the method of proof of Theorem~\ref{thm:main-rank1}. Indeed, by Cornut--Vatsal \cite{cornut-vatsal}, regardless of the order of vanishing of $L(E/K,s)$ at $s=1$, the Heegner points $z_n:=\pi(\iota_1(x_1(p^n)))\in E(H_{p^n})\otimes_{\mathbf{Z}}\mathbf{Z}_p$ are non-torsion for $n\gg 0$. For such $n$, letting 
\[
z_{n,\chi}\in E(H_{p^n})^\chi\subset E(H_{p^n})\otimes_{\mathbf{Z}[\mathrm{Gal}(H_{p^n}/K)]}\mathbf{Z}_p[\chi]
\] 
be the image of $y_n$ in the $\chi$-isotypical component for a primitive character $\chi:\mathrm{Gal}(H_{p^n}/K)\rightarrow\mathbf{Z}[\chi]^\times$, the Gross--Zagier formula \cite{yuan-zhang-zhang} combined with W.~Zhang's work \cite{wei-zhang-mazur-tate} and a generalization of Kolyvagin's structure theorem  for Tate--Shafarevich  groups should yield an analogue of $(\ref{eq:ppartBSD})$ in terms of the index of 
$z_{n,\chi}$ in $E(K_n)^\chi$. 

With these results in hand, to remove the analytic rank one hypothesis from Theorem~\ref{thm:main-rank1} it would suffice to generalize our reduction of Conjecture~\ref{conj:HPMC} to the corresponding analogue of (\ref{eq:ppartBSD}). This would provide an alternate proof of our main results in this paper (namely, Theorem~\ref{thm:main} and Theorem~\ref{thm:main-BDP}) without the need to assume that $p$ is non-anomalous.
\end{rem}

\bibliographystyle{amsalpha}
\bibliography{bck}

\end{document}